\documentclass{article}

\usepackage[T1]{fontenc}
\usepackage{geometry}
\usepackage{amssymb,amsmath,amsthm}
\usepackage{cases}
\usepackage{enumitem}
\usepackage{xcolor}
\usepackage{comment}
\usepackage{graphicx}\graphicspath{{img/}} 
\usepackage{textcomp}
\usepackage{mathrsfs}
\usepackage{bbm}
\usepackage{bm}
\usepackage{caption}
\usepackage{stmaryrd}
\usepackage{url}
\usepackage{calc}
\usepackage{ifthen}
\usepackage{psfrag}
\usepackage[active]{srcltx}   
\usepackage{ytableau}\ytableausetup{smalltableaux}
\usepackage{mathtools}
\usepackage{todonotes}
\usepackage{pdfrender}
\usepackage[colorlinks=true,urlcolor=magenta!60!black,citecolor=green!50!black,linkcolor=red!90!black,bookmarksnumbered=true]{hyperref}

\theoremstyle{plain}
\newtheorem{thm}{Theorem}
\newtheorem{lem}[thm]{Lemma}
\newtheorem{prop}[thm]{Proposition}

\newtheorem{corol}[thm]{Corollary}

\theoremstyle{definition}
\newtheorem{rem}{Remark}

\newcommand{\de}{\mathrel{\mathop:}\hspace*{-.6pt}=}

\newcommand{\wh}{\widehat}

\newcommand{\sew}{\,{\Join}\,}
\newcommand{\Orb}{\mathrm{Orb}}

\newcommand{\m}{\mathbf{m}}
\newcommand{\bq}{\mathbf{q}}
\newcommand{\br}{\mathbf{r}}
\newcommand{\bt}{\mathbf{t}}

\newcommand{\bx}{\mathbf{x}}
\newcommand{\by}{\mathbf{y}}
\newcommand{\bl}{\bm{\ell}}
\newcommand{\bgamma}{\boldsymbol{\gamma}}
\newcommand{\bgammai}{\boldsymbol{\gamma}_{\mathrm{i}}}

\newcommand{\vertex}{\bullet}
\newcommand{\edge}{%
	\begin{tikzpicture}
	\draw (0,0)--(.16,.08); \filldraw (0,0) circle (.2pt); \filldraw (.16,.08) circle (.2pt);
	\fill (.11,.055)--(.06,.055)--(.08,.015)--cycle;
	\end{tikzpicture}}
\newcommand{\iedge}{%
	\begin{tikzpicture}
	\draw (0,0)--(.16,.08); \filldraw (0,0) circle (.2pt); \filldraw (.16,.08) circle (.2pt);
	\fill (.11,.055)--(.06,.055)--(.08,.015)--cycle;
	\draw (.08,0)--(0.12,-0.04);\draw (.08,-0.04)--(0.12,0);
	\draw (0.024,0.072)--(0.064,0.112);\draw (0.024,0.112)--(0.064,0.072);
	\end{tikzpicture}}
\newcommand{\redge}{%
	\begin{tikzpicture}
	\draw (0,0)--(.16,.08); \filldraw (0,0) circle (.2pt); \filldraw (.16,.08) circle (.2pt);
	\fill (.11,.055)--(.06,.055)--(.08,.015)--cycle;
	\draw (.08,0)--(0.12,-0.04);\draw (.08,-0.04)--(0.12,0);
	\end{tikzpicture}}
\newcommand{\bedge}{\partial}
\newcommand{\xs}{\mathrm{x}}
\newcommand{\B}[1]{\scalebox{1.5}{#1}}

\newcommand{\cT}{\mathcal{T}}
\newcommand{\cTv}{\cT^{\vertex}}
\newcommand{\cTe}{\cT^{\edge}}
\newcommand{\cTbe}{\cT^{\bedge\edge}}
\newcommand{\cB}{\mathcal{B}}
\newcommand{\cBv}{\cB^{\vertex}}
\newcommand{\cBi}{\cB^{\iedge}}
\newcommand{\cBbi}{\cB^{\bedge\iedge}}
\newcommand{\cBbr}{\cB^{\bedge\redge}}
\newcommand{\cS}{\mathcal{S}}
\newcommand{\cSv}{\cS^{\vertex}}
\newcommand{\cSr}{\cS^{\redge}}
\newcommand{\cSbr}{\cS^{\bedge\redge}}
\newcommand{\cX}{\mathcal{X}}
\newcommand{\cXv}{\cX^{\vertex}}

\newcommand{\cXbr}{\cX^{\bedge\redge}}
\newcommand{\cXbi}{\cX^{\bedge\iedge}}
\newcommand{\cXx}{\cX^{\xs}}
\newcommand{\cXbx}{\cX^{\bedge \xs}}
\newcommand{\cQ}{\mathcal{Q}}
\newcommand{\cQf}{\cQ^{\triangle}}
\newcommand{\cQL}{\cQ^{\mathrm{L}}}

\newcommand{\cF}{\mathcal{F}}
\newcommand{\cG}{\mathcal{G}}
\newcommand{\cO}{\mathcal{O}}
\newcommand{\cY}{\mathcal{Y}}
\newcommand{\cZ}{\mathcal{Z}}
\newcommand{\rN}{\mathrm{N}}
\newcommand{\rS}{\mathrm{S}}
\newcommand{\Cat}{\mathrm{Cat}}
\newcommand{\Nar}{\mathrm{Nar}}

\definecolor{vert}{rgb}{0,0.666,0}
\definecolor{pole}{RGB}{184,113,255}
\definecolor{stepcol}{rgb}{0,.44,.22}

\usepackage{contour}
\contourlength{0.6pt}
\newcommand{\unli}[1]{\underline{\phantom{\smash{\textbf{#1}}}}%
	\llap{\textbf{#1}}%
	}

\makeatletter
\renewcommand\paragraph{\@startsection{paragraph}{4}{\z@}%
	{3.25ex \@plus1ex \@minus.2ex}%
	{-1em}%
	{\normalfont\normalsize\bfseries\unli}}
\makeatother

\title{Slit-slide-sew bijections for oriented planar maps}
\author{J\'er\'emie Bettinelli\thanks{CNRS \& Laboratoire d'Informatique de l'\'Ecole polytechnique. Partially supported by ANR-20-CE48-0018 \emph{3DMaps}, ANR-21-CE48-0007 \emph{IsOMa}, ANR-23-CE48-0018 \emph{CartesEtPlus}.}%
	\and %
	Éric Fusy\thanks{CNRS \& Laboratoire d'Informatique Gaspard Monge. Partially supported by ANR-20-CE48-0018 \emph{3DMaps}, ANR-23-CE48-0018 \emph{CartesEtPlus}.}%
	\and %
	Baptiste Louf\thanks{CNRS \& Institut de Math\'ematiques de Bordeaux. Partially supported by ANR-23-CE48-0018 \emph{CartesEtPlus}.}%
}

\let\oldemph\emph
\renewcommand{\emph}[1]{\textcolor{red!65!black}{\oldemph{#1}}}

\begin{document}
\maketitle

\begin{abstract}
We construct growth bijections for bipolar oriented planar maps and for Schnyder woods. These give direct combinatorial proofs of several counting identities for these objects.

Our method mainly uses two ingredients. First, a slit-slide-sew operation, which consists in slightly sliding a map along a well-chosen path. Second, the study of the orbits of natural rerooting operations on the considered classes of oriented maps.
\end{abstract}


\section{Introduction}

\subsection{Growth bijections for trees and maps}

The main theme of this paper is that of \emph{growth bijections} for combinatorial structures: in simple words, given a combinatorial family $(\cF_{n})_{n\geq 0}$ indexed by a size parameter~$n$, a growth bijection is a procedure that bijectively constructs all the objects of~$\cF_{n}$ from those of~$\cF_{n-1}$, using only a slight modification. For this to be possible, the objects of the smaller size ($n-1$) bear some additional markings, and those of the larger size ($n$) also bear different additional markings. Such growth bijections are in general inspired by pre-existing simple recursive counting identities linking the cardinalities~$|\cF_{n}|$ and~$|\cF_{n-1}|$.

A famous instance of a growth bijection is \emph{R\'emy's bijection}~\cite{Rem85} (see also~\cite{Rod38}), which provides a proof of the identity
\begin{equation}\label{remeq}
(n+1)\,\Cat_{n}=2\,(2n-1)\,\Cat_{n-1},
\end{equation}
where $\Cat_n=\frac{1}{n+1}\binom{2n}{n}$ is the $n$-th Catalan number, which, in particular, counts binary trees on $n+1$ leaves (and thus with $2n+1$ edges). As a result, this identity~\eqref{remeq} can be interpreted as an equinumerosity between the sets of
\begin{itemize}
	\item binary trees on $n+1$ leaves with a marked leaf;
	\item binary trees on $n$ leaves with a marked edge, as well as a parameter either $\mathrm{left}$ or $\mathrm{right}$.
\end{itemize}
Consequently, there exists a bijection between these two sets and the idea is to find a simple constructive one, in the sense that the modification needed to pass from one set to the other one is as minimal as possible. In this example, the bijection is quite straightforward to find and we invite the reader unknowing of it to find it as a warm-up exercise.

\bigskip
For plane forests and planar maps, a family of growth bijections relying on \emph{slit-slide-sew} operations has been introduced by the first author~\cite{Bet14,Bet20,BeKo26}. Such an operation also appears (under the name ``cut-and-slide'') in a work by the third author~\cite{Lou19} that extends R\'emy's bijection to all planar maps, and in a work of Schabanel~\cite{Sch25} on bipartite maps via a bijection of Schaeffer~\cite{Sch97}. We must also mention that growth bijections have inspired \emph{probabilistic growth schemes} both for trees and planar maps~\cite{LuWi04,ABe14,Fle24}.

\subsection{Families of oriented maps}\label{secmifa}

In the present paper, we will focus on three specific classes of oriented planar maps, namely:
\begin{itemize}
	\item bipolar oriented quasi-triangulations;
	\item bipolar oriented maps;
	\item Schnyder woods.
\end{itemize}
We will show that slit-side-sew bijections can be successfully applied to these families, yielding a simple proof to a counting formula in each case (Propositions~\ref{propt}, \ref{propb} and~\ref{props}). However, contrary to previous bijections of this type, an additional ingredient is needed here: we will have to compute the probability that a given edge possesses a certain ``good'' property. This will be done bijectively by averaging over orbits of a natural rerooting operation.

\paragraph{Bipolar oriented maps.}
The considered maps are all specific classes of bipolar oriented maps; for Schnyder woods it reduces to such a class via a known easy bijection. First, a \emph{planar map} is an embedding of a finite connected graph (possibly with multiple edges and loops) into the sphere, considered up to orientation-preserving homeomorphisms. Then a \emph{bipolar oriented map} is a planar map whose edges are all oriented as follows. One oriented edge is distinguished and called the \emph{root edge}. Its tail and head are respectively called the \emph{South Pole} and \emph{North Pole}, and denoted by~$\rS$ and~$\rN$. The South Pole is a \emph{source}, which means that all the edges incident to it are \emph{outgoing}, that is, oriented away from it; the North Pole is a \emph{sink}, which means that all the edges incident to it are \emph{incoming}, that is, oriented toward it; every non-pole vertex is neither a source nor a sink, that is, is incident to at least one incoming and one outgoing edge. Finally, there are no directed cycles. See Figure~\ref{bom}.

The \emph{external face} is taken as the face on the left of the root edge. In figures, we will always draw the external face as the unbounded component of the plane, in white; we will use a yellowish coloring for internal faces. The vertices incident to the external face make up the \emph{boundary} of the map and are called \emph{external vertices}. The other vertices are called \emph{internal vertices}.

\begin{figure}[ht!]
	\centering\includegraphics[width=11cm]{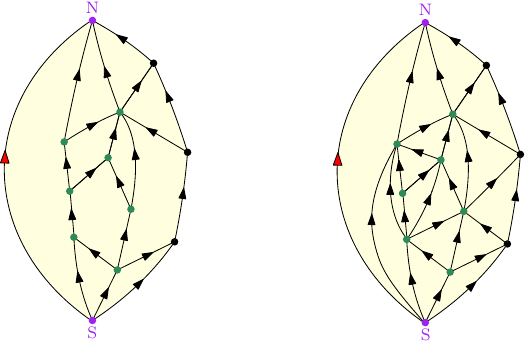}
	\caption{\textbf{Left.} A bipolar oriented map. The poles are in purple; the root is in red; the internal vertices are in green. \textbf{Right.} A bipolar oriented quasi-triangulation: all the internal faces have degree~$3$.}
	\label{bom}
\end{figure}

\paragraph{Bipolar oriented quasi-triangulations.}
A \emph{quasi-triangulation} is a rooted map where the external face has simple contour, of any degree larger than or equal to~$2$, and the internal faces all have degree $3$. A \emph{bipolar oriented quasi-triangulation} is a bipolar map whose underlying map is a quasi-triangulation.

\paragraph{Local rules.}
It is well known (see e.g.~\cite{FrOsRo95}) that a bipolar oriented map satisfies the following local properties, illustrated in Figure~\ref{bomrules}.
\begin{itemize}
	\item Every non-pole vertex has its incident edges partitioned into a nonempty group of consecutive incoming edges and a nonempty group of consecutive outgoing edges.
	\item Every internal face has its contour partitioned into two directed paths sharing the same extremities: a left lateral path (having the face on its right) and a right lateral path (having the face on its left). The \emph{left length} (resp.\ \emph{right length}) of an internal face is the length of its left (resp.\ right) lateral path.
	\item The path formed by the right outer boundary is a directed path from the source to the sink.
\end{itemize}
Moreover these local properties easily guarantee acyclicity, hence actually characterize bipolar oriented maps.

\begin{figure}[ht!]
	\centering
	\includegraphics[width=12cm]{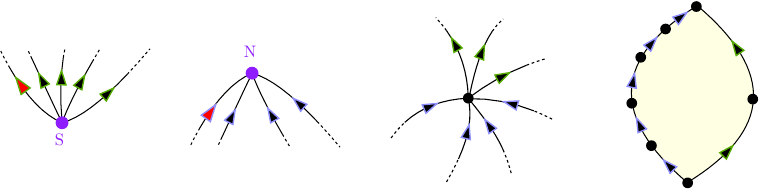}
	\caption{The local rules around the vertices and faces of a bipolar oriented map. Here, the right length of the face is~$2$.}
	\label{bomrules}
\end{figure}

\paragraph{Schnyder woods.}
We consider a triangulation~$\bt$, that is, a planar map whose faces are all of degree~$3$, with a distinguished face, drawn as the unbounded face in the plane. We denote by~$\rho_1$, $\rho_2$, $\rho_3$ the incident vertices, in clockwise order. A \emph{Schnyder wood}~\cite{Sch89} of~$\bt$ is a partition of its internal edges into three trees~$\bt_1$, $\bt_2$, $\bt_3$ satisfying the following two conditions. See Figure~\ref{schnyd}.
\begin{enumerate}[label=\textit{(\arabic*)}]
	\item For every $i\in\{1,2,3\}$, the tree~$\bt_i$ spans all the internal vertices and the external vertex~$\rho_{i}$, which is its root.
	\item\label{shnydrule2} If we orient the trees toward their roots, one has the local rule depicted on the right side of Figure~\ref{schnyd} around each internal vertex, namely: in clockwise direction, one always sees the outgoing edge of~$\bt_1$, the incoming edges of~$\bt_3$, the outgoing edge of~$\bt_2$, the incoming edges of~$\bt_1$, the outgoing edge of~$\bt_3$, and finally the incoming edges of~$\bt_2$. For each tree, the number of incoming edges might be null.
\end{enumerate}

\begin{figure}[ht!]
	\centering\includegraphics[height=6cm]{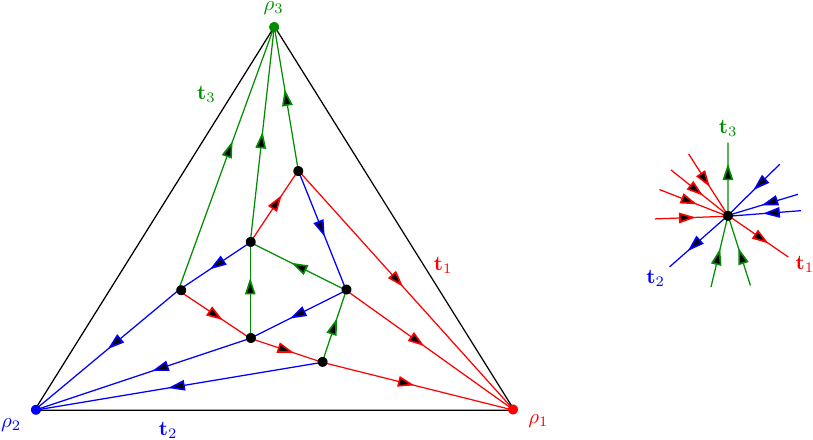}
	\caption{\textbf{Left.} A Schnyder wood on~$9$ vertices. The edges of three trees are oriented toward the respective tree roots. \textbf{Right.} The local rule around an internal vertex.}
	\label{schnyd}
\end{figure}

It is well known~\cite{FuPoSc09} that Schnyder woods are in bijection with bipolar oriented maps such that every internal face has right length~$2$. The bijection is recalled and illustrated in Figure~\ref{schnydbij}. 

\begin{figure}[ht!]
	\centering\includegraphics[height=6cm]{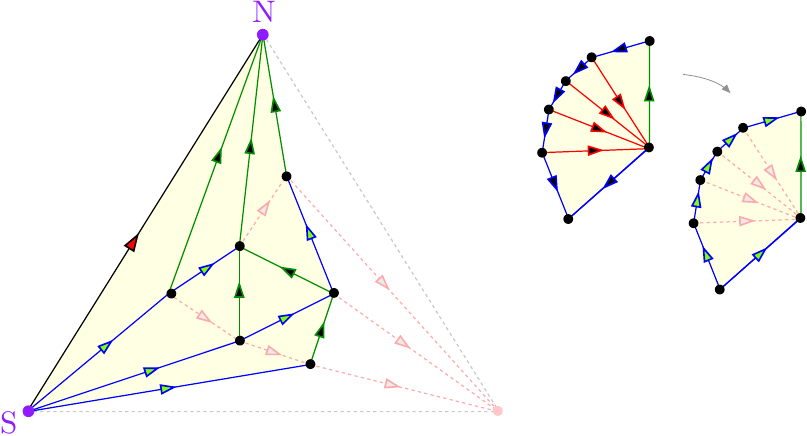}
	\caption{\textbf{Left.} To obtain the bipolar oriented map from the Schnyder wood, remove the edges of~$\bt_1$, the vertex~$\rho_1$ and the two external edges incident to~$\rho_1$, revert the edges of~$\bt_2$, set $\rS\de\rho_2$, $\rN\de\rho_3$, and the root as the remaining external edge, oriented from~$\rS$ to~$\rN$. Observe that the degree of~$\rho_1$ becomes the degree of the external face. \textbf{Right.} The internal vertices of the Schnyder wood naturally correspond to the faces of the bipolar oriented map and these all have right length~$2$.}
	\label{schnydbij}
\end{figure}

\subsection{Enumeration}\label{secenum}

\paragraph{Counting formulas.}
We consider the following counting coefficients:
\begin{itemize}
	\item the number $T_{k,j}$ of bipolar oriented \emph{quasi-triangulations} with~$k$ internal vertices and an external face of degree~$j$, for any $k\geq 0$, $j\geq 2$\,;
	\item the number $B_{k,\ell,j}$ of bipolar oriented \emph{maps} with~$k$ internal vertices, $\ell$ internal faces, and an external face of degree~$j$, for any $k\geq 0$, $\ell\geq 1$, $j\geq 2$\,;
	\item the number $S_{k,j}$ of Schnyder woods on $k+j+1$ vertices and such that~$\rho_1$ has degree~$j$, for any $k\geq 0$, $j\geq 3$. Equivalently, by the bijection of Figure~\ref{schnydbij}, $S_{k,j}$ is the number of bipolar oriented maps with~$k$ internal vertices, an external face of degree~$j$, and such that all internal faces have right length~$2$.
\end{itemize}

Note in particular that $B_{k,\ell,2}$ is the the number of bipolar oriented maps with $k+2$ vertices
and~$\ell$ faces (by merging the left boundary of length~$1$ to the root edge), and that $S_{k,3}$ is the number of Schnyder woods on $k+3$ vertices (by removing the degree~$3$ external vertex~$\rho_1$ and its incident edges, and then taking the incident internal vertex as the new vertex~$\rho_1$). 

These coefficients are given by the following explicit formulas:
\begin{align}
T_{k,j} 		& =j\,(j-1)\frac{(3k+2j-4)!}{k!\, (k+j-1)!\, (k+j)!}\,
				&\text{$k\geq 0$, $j\geq 2$\,;}\label{nbt}\\[2mm]
B_{k,\ell,j} 	& =j\,(j-1)\frac{(k+\ell-2)!\,(k+\ell+j-2)!\,(k+\ell+j-3)!}{k!\,(k+j)!\,(k+j-1)!\,\ell!\,(\ell-1)!\,(\ell-2)!}\,
				&\text{$k\geq 0$, $\ell\geq 1$, $j\geq 2$\,;}\label{nbb}\\[2mm]
S_{k,j} 		& =j\,(j-1)\,(j-2)\frac{(2k+2j-4)!\,(2k+j-3)!}{k!\,(k+j)!\, (k+j-1)!\, (k+j-2)!}\,
				&\text{$k\geq 0$, $j\geq 3$\,.}\label{nbs}
\end{align}

The first two formulas are obtained in~\cite{BMo11} via a recursive decomposition of bipolar oriented maps and an application of the obstinate kernel method. They can also be obtained from known bijections: in~\cite{Bor17,KeMiShWi19} for the first formula, and in~\cite{AlPo15,FuPoSc09} for the second one. The third one can be obtained from the bijection in~\cite{BeBo09}. See Section~\ref{sec:tableau} for more details on how these bijections yield the formulas.   


\paragraph{Combinatorial identities.}
From the above formulas, we easily obtain the following ``growth identities'', for which we will give a bijective interpretation in the present work.

\begin{prop}[Bipolar oriented quasi-triangulations]\label{propt}
Let $k\ge 1$ and $j\ge 2$. Then
\begin{align}\label{eqt}
\mathllap{k\,T_{k,j}} &= \mathrlap{\left( 1-\frac{2}{j+1} \right)\big(3k+2j-4\big)\,T_{k-1,j+1}\,.}\hspace{55mm}
\end{align}
\end{prop}

\begin{prop}[Bipolar oriented maps]\label{propb}
Let $k\ge 1$, $\ell\ge 1$, and $j\ge 2$. Then the following identity holds:
\begin{align}\label{eqb}
\mathllap{k\,B_{k,\ell,j}} &= \mathrlap{\left( 1-\frac{2}{j+1} \right)\big(k+\ell-2\big)\,B_{k-1,\ell,j+1}\,.}\hspace{55mm}
\end{align}
\end{prop}

\begin{prop}[Schnyder woods]\label{props}
Let $k\ge 1$ and $j\ge 3$. Then the following identity holds:
\begin{align}\label{eqs}
\mathllap{k\,S_{k,j}} &= \mathrlap{\left( 1-\frac{3}{j+1} \right)\big(2k+j-3\big)\,S_{k-1,j+1}\,.}\hspace{55mm}
\end{align}
\end{prop}

\paragraph{Bijective interpretation.}
The three combinatorial identities~\eqref{eqt}, \eqref{eqb}, \eqref{eqs} all possess the same structure (seeing the third one as counting bipolar oriented maps with internal faces of right length $2$). The structure is as follows: 
\begin{itemize}
	\item On the left-hand side, there is the coefficient with parameters~$k$ and~$j$, together with the prefactor~$k$, which counts internal vertices.
	\item On the right-hand side, there is the coefficient with parameters $k-1$ and $j+1$, together with two prefactors:
	\begin{enumerate}[label=(\textit{\alph*})]
		\item\label{prefa} the first one has the form $\big( 1-\frac{\lambda}{j+1} \big)$, where $j+1$ is the external face degree;
		\item\label{prefb} the second one counts (most) edges.		
	\end{enumerate}
\end{itemize}
We will use a unified framework. To fit this framework, we will need to adopt some slightly different points of view depending on the class of maps into consideration. This bears the consequence that the prefactor~\ref{prefb} has a slightly different interpretation for each class.
More precisely, using Euler's characteristic formula, one easily obtains that the maps counted by $B_{k-1,\ell,j+1}$ all possess $k+j+\ell-1$ edges. Furthermore, in the particular case of quasi-triangulations, one obtains that the maps have $\ell=2k+j-3$ internal faces, and thus $3k+2j-4$ edges. In the case of bipolar oriented maps corresponding to Schnyder woods, the condition on the right lengths yields that the number of edges is $2\ell+1$, and thus $2k+2j-3$. As a result, the prefactor~\ref{prefb} counts
\begin{itemize}
	\item \textbf{all the edges} in~\eqref{eqt} on quasi-triangulations;
	\item all but $j+1$ edges, that is, \textbf{edges having internal faces on both sides} in~\eqref{eqb} on general bipolar oriented maps;
	\item all but~$j$ edges, that is, \textbf{edges having an internal face on their right} in~\eqref{eqs} on bipolar oriented maps corresponding to Schnyder woods.
\end{itemize}

The prefactor~\ref{prefa} will be seen as a probability that an edge from the class-specific proper set of edges satisfies an extra property called \emph{boundary-reaching} and whose definition will be given in Section~\ref{secedge}. It will be shown that the set of maps we consider can be partitioned into orbits of a rerooting operation, and that the probability given by~\ref{prefa} holds over each orbit. This is an instance of the so-called \emph{homomesy} phenomenon: see Remark~\ref{rem:homomesy}. 

Summing up, in the proper class of bipolar oriented maps, the left-hand side counts maps carrying a distinguished internal vertex and the right-hand side counts maps with the same number of vertices in total but one more on the boundary, carrying a distinguished edge satisfying proper constraints.

\subsection{Organization of the paper}

The aim of the present paper is to exhibit a bijective interpretation of the identities~\eqref{eqt}, \eqref{eqb}, \eqref{eqs}. We start by presenting in Section~\ref{secsss} our bijection and its specializations to the three specific classes of bipolar oriented maps we consider here. In Section~\ref{secbr}, we come back to the enumerative consequences of the bijections and compute the probabilities~\ref{prefa} that edges are boundary-reaching. 

Since our method provides alternate proofs to the identities~\eqref{eqt}, \eqref{eqb}, \eqref{eqs}, we also recover the counting formulas~\eqref{nbt}, \eqref{nbb}, \eqref{nbs}, from the easy to obtain initial conditions postponed in Section~\ref{secbase}.

\paragraph{Acknowledgments.} Special thanks are given to the funding project ANR-23-CE48-0018 \emph{CartesEtPlus} and its participants, since this work originated from its launching gathering. We also thank Cl\'ement Chenevi\`ere for pointing to the link with homomesy.

\section{Slit-slide-sew bijection}\label{secsss}

In this section, we only consider bipolar oriented maps. As a consequence, all the edges are oriented. For an edge~$e$, we respectively denote its tail and head by~$e^-$ and~$e^+$.

\paragraph{Paths.}
A \emph{path} from a vertex~$v$ to a vertex~$v'$ is a finite sequence $\bgamma=(e_1,e_2,\dots,e_k)$ of edges such that $e_1^-=v$, for $1\le i \le k-1$, $e_i^+=e_{i+1}^-$, and $e_k^+=v'$. 
Beware that a path is only made of edges oriented in the underlying orientation; they cannot be used ``backward''. Also, since we consider here acyclic orientations, paths are \emph{simple}, in the sense that the vertices they visit are all distinct.

Among all the paths from a vertex~$v$ to the North Pole~$\rN$, one will be of particular interest. Recall that, around a given non-pole vertex, the edges are partitioned into a consecutive group of incoming edges then a consecutive group of outgoing edges. The \emph{rightmost path} from~$v$ to~$\rN$ is the path that takes at each vertex the rightmost outgoing edge, that is, the first outgoing edge in counterclockwise order after the group of incoming edges. If~$v$ is the South Pole~$\rS$, then the first edge of the path is the one having the external face to its right. 

Note that, as soon as the rightmost path~$\bgamma$ from a vertex~$v$ to~$\rN$ reaches the boundary of the map, it stays on it until the Pole~$\rN$. Thus, $\bgamma$ is made of an \emph{internal part}~$\bgammai$, followed (once the boundary is reached) by an \emph{external part}, which is included in the boundary. Clearly, the internal part has length~$0$ if and only if~$v$ is external. We call \emph{external index} with respect to~$v$ the length of the external part.

\subsection{The construction, from maps with a distinguished internal vertex}\label{secbijve}

Let~$\cB$ denote the set of bipolar oriented maps, and~$\cBv$ denote the set of bipolar oriented maps carrying a distinguished internal vertex, that is, the set of pairs $(\m,v)$, where $\m\in\cB$ and~$v$ is an internal vertex of~$\m$. For such a pair, the \emph{marked face} is defined as the internal face~$f$ at the right of~$v$, that is, the one at the right of the rightmost outgoing edge at~$v$. The \emph{external index}, denoted by~$\delta$, is the external index with respect to~$v$. 

Let $(\m,v)\in\cBv$. We break down the process into the following steps. See Figure~\ref{bij_Phi} (from top left to top right in counterclockwise order).

\begin{figure}[ht!]
	\centering\includegraphics[width=14.5cm]{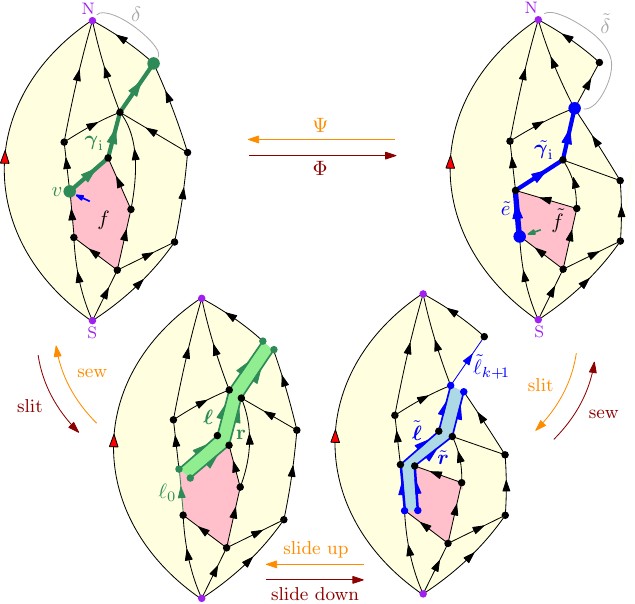}
	\caption{The slit-slide-sew bijections~$\Phi$, $\Psi$ between bipolar oriented maps with a marked internal vertex and bipolar oriented maps with a marked right-internal boundary-reaching edge. The little arrows show at which corner to enter in the slitting step. The marked face is highlighted in red and sees its left length decrease or increase by one.}
	\label{bij_Phi}
\end{figure}

\begin{description}[style=nextline]
\item[\textcolor{stepcol}{A. Sliding path}]
	We consider the internal part $\bgammai=(e_1,e_2,\dots,e_k)$ of the rightmost path from~$v$ to~$\rN$. We let~$\ell_0$ be the edge following~$e_1$ in clockwise order around~$v$. Since~$v$ is internal and~$e_1$ is the rightmost outgoing edge of~$v$, the edge~$\ell_0$ is incoming at~$v$, with the marked face~$f$ at its right. 

\item[\textcolor{stepcol}{B. Slitting, sliding, sewing}]
	We slit along~$\bgammai$, entering at~$v$ from the corner in~$f$ and exiting through the external face. This doubles the path~$\bgammai$, making up two copies: $\bl=(\ell_1,\dots,\ell_k)$ to the left, and $\br=(r_1,\dots,r_k)$ to the right.

	We then sew back~$\bl$ onto~$\br$ after \textbf{sliding down} the right side by one unit, in the sense that we match~$\ell_{i-1}$ with~$r_{i}$, for every $1\le i \le k$. For further reference, we denote by $\ell_{i-1}\sew r_{i}$ the resulting edge.

\item[\textcolor{stepcol}{C. Output}]
	In the resulting map~$\tilde\m$, we set $\tilde e\de \ell_0\sew r_1$. The output of the construction is the pair $\Phi(\m,v)\de(\tilde\m,\tilde e)$.
\end{description}

Before presenting the reverse bijection, let us first see what properties the output must satisfy.

\subsection{Edge constraints}\label{secedge}

An edge may satisfy the following properties, which are crucial for our purposes.

\paragraph{Boundary-reaching edge.}
An edge~$e$ in a bipolar oriented map~$\m$ is \emph{boundary-reaching} if the rightmost path from~$e^+$ to the North Pole~$\rN$ reaches the boundary of~$\m$ before~$\rN$, that is, if the external index with respect to~$e^+$ is strictly positive. Equivalently, by planarity, $e$ is \textbf{not} boundary-reaching if all the paths from~$e^+$ to~$\rN$ reach the boundary at~$\rN$. See Figure~\ref{bombr}.

\begin{figure}[ht!]
	\centering\includegraphics[height=6cm]{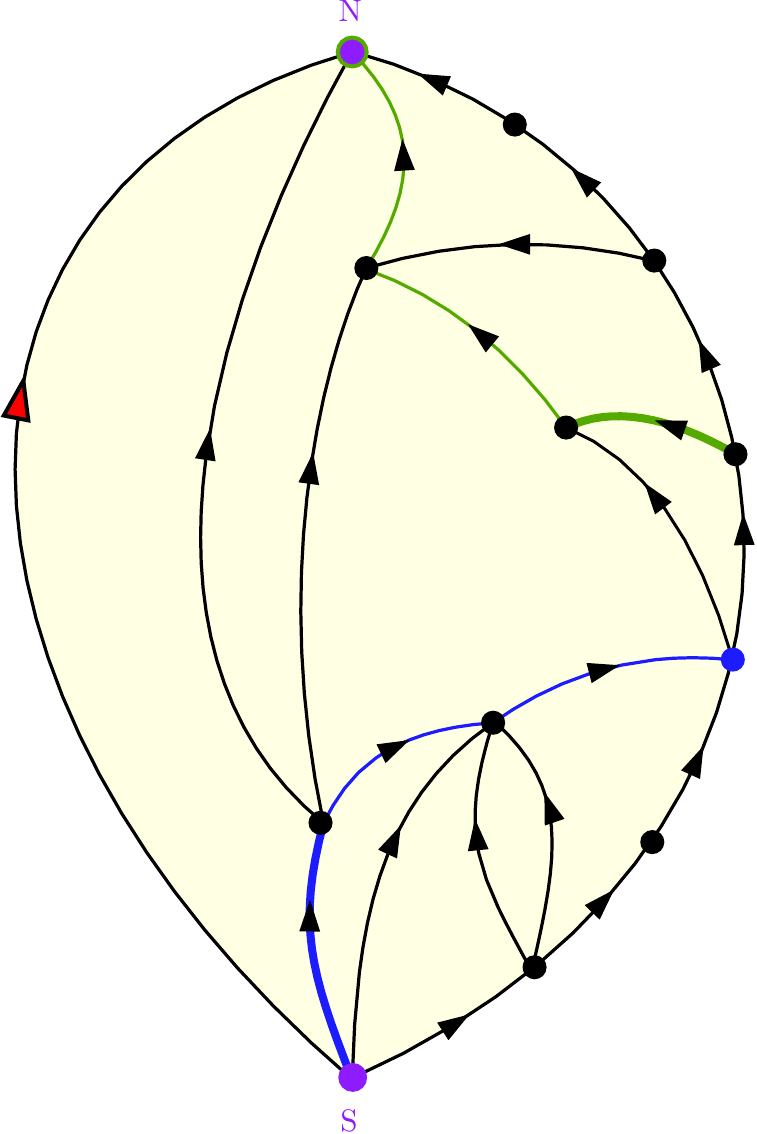}
	\caption{The thick blue edge is boundary-reaching whereas the thick green one is not. The rightmost paths from their ends to~$\rN$ reach the boundary at the vertices with corresponding color.}
	\label{bombr}
\end{figure}

\paragraph{Right-internal edge.}
An edge~$e$ in a bipolar oriented map~$\m$ is called \emph{right-internal} if the face to its right is an internal face. Note that, except in the trivial case of a map with no internal faces, the root edge is always right-internal. In fact, all the edges are right-internal, except those on the right part of the boundary.

\begin{prop}\label{meinBbr}
Let $(\m,v)\in\cBv$, of external index $\delta\geq 0$, and let $\Phi(\m,v)\de(\tilde\m,\tilde e)$. Then~$\tilde\m$ is a bipolar oriented map, with~$\tilde e$ a right-internal edge. Moreover, the external index of~$\tilde\m$ with respect to~$\tilde e^+$ is equal to $\delta+1$, so that, in particular, $\tilde e$ is boundary-reaching. 
\end{prop}

\begin{proof}
We rely on the characterization of bipolar oriented maps by the local rules given in Section~\ref{secmifa}. The local condition at any internal faces still holds, the left length and right length being preserved, except for the marked face~$f$, whose left length decreases by~$1$. Note in passing that~$f$ is at the right of~$\tilde e$, so that~$\tilde e$ is right-internal.  The local condition for the external boundary also still holds.

The only nonstraightforward property remaining to check for~$\tilde\m$ is that the local rules still hold at vertices along the sliding path. Note that the edges to the right of~$\bgammai$ are directed toward~$\bgammai$ (by definition of rightmost paths). Hence, for each vertex along~$\bgammai$, we replace, between an incoming and an outgoing edge, a (possibly empty) consecutive group of incoming edges with an other (possibly empty) consecutive group of incoming edges.

Besides concluding the proof that~$\tilde\m$ is a bipolar oriented map, this property also ensures that the (possibly empty, in the case where $k=1$) path $(\ell_1\sew r_2,\dots, \ell_{k-1}\sew r_k)$ is prefix to the rightmost path from~$\tilde e^+$ to~$\rN$ in $\tilde\m$. 
These edges preserve the property of having an internal face on each side, hence they form a path of internal edges in $\tilde\m$, ending at the origin of~$\ell_k$, which becomes an external edge after the sewing operation. Hence, $\tilde\bgammai:=(\ell_0\sew r_1,\dots, \ell_{k-1}\sew r_k)$ is made of~$\tilde e$ followed by the internal part of the rightmost path from~$\tilde e^+$ to~$\rN$, while the edge~$\ell_k$ has been transferred to its external part. The external index of~$\tilde\m$ with respect to~$\tilde e^+$ is thus $\delta+1$, as claimed.
\end{proof}

\subsection{The construction, from maps with a distinguished proper edge}\label{secbijev}

We now present the inverse construction. Let~$\cBbr$ denote the set of bipolar oriented maps carrying a distinguished boundary-reaching right-internal edge, and let $(\tilde\m,\tilde e)\in\cBbr$. For such a pair, the \emph{marked face} is the internal face~$\tilde f$ at the right of~$\tilde e$, and the \emph{external index}, denoted by~$\tilde\delta$, is the external index with respect to~$\tilde e^+$. Note that $\tilde\delta\geq 1$ since~$\tilde e$ is boundary-reaching. 

As above, we break down the process into~3 steps. See Figure~\ref{bij_Phi}, to be read from top right to top left in clockwise order.

\begin{description}[style=nextline]
\item[\textcolor{stepcol}{1. Sliding path}]
We consider the path $\tilde\bgammai=(\tilde e_1=\tilde e,\tilde e_2,\dots,\tilde e_k)$ consisting of~$\tilde e$ followed by the internal part of the rightmost path from~$\tilde e^+$ to~$\rN$.

\item[\textcolor{stepcol}{2. Slitting, sliding, sewing}]
	We slit along~$\tilde\bgammai$, entering at~$\tilde e^-$ from the corner in the marked face~$\tilde f$ and exiting through the external face. This doubles the path~$\tilde \bgammai$, making up two copies, one to the left, $\tilde \bl=(\tilde\ell_1,\dots,\tilde \ell_k)$, and one to the right, $\tilde\br=(\tilde r_1,\dots,\tilde r_k)$. We also let $\tilde\ell_{k+1}$ be the external edge whose origin is the end of~$\ell_k$ (this edge exists since $\tilde\delta\geq 1$).

	We then sew back~$\tilde\bl$ onto~$\tilde\br$ after \textbf{sliding up} the right side by one unit, in the sense that we match~$\tilde\ell_{i+1}$ with~$\tilde r_{i}$, for every $1\le i \le k$.

\item[\textcolor{stepcol}{3. Output}]
	In the resulting map~$\m$, we denote by~$v$ the vertex~$\tilde \ell_1^+=\tilde r_1^-$. The output of the construction is the pair $\Psi(\tilde\m,\tilde e)\de(\m,v)$.
\end{description}

\begin{prop}\label{mvinBv}
For any $(\tilde\m,\tilde e)\in\cBbr$, the output $(\m,v)=\Psi(\tilde\m,\tilde e)$ of the above construction is in~$\cBv$. Moreover, if $\tilde\delta\geq 1$ denotes the external index of $(\tilde\m,\tilde e)$, then the external index of $(\m,v)$ is $\tilde\delta -1$.
\end{prop}

\begin{proof}
Very similarly to the proof of Proposition~\ref{meinBbr}, the construction preserves the local condition of bipolar oriented maps (left and right lengths of internal faces are preserved, except for the left length of the marked face that increases by~$1$).

The sewn path is the internal part of the rightmost path from~$v$ to~$\rN$. Moreover, there is one less edge (the edge $\tilde\ell_{k+1}$) in the external part, in comparison with the rightmost path from~$\tilde e^+$ to~$\rN$ in the original map~$\tilde\m$. Note also that the marked face~$\tilde f$ becomes the face at the right of~$v$ in~$\m$, so that~$v$ has to be an internal vertex. 
\end{proof}

\subsection{Bijections and specializations}

\paragraph{Bijections.}
We now establish that the mappings~$\Phi$ and~$\Psi$ are indeed bijections, which are inverse to one another. 

\begin{thm}\label{thmsss}
The slit-slide-sew mappings $\Phi\colon \cBv \to \cBbr$ and $\Psi\colon \cBbr \to \cBv$ are inverse bijections. Furthermore, the mapping~$\Phi$
\begin{itemize}
	\item lets the external index and the external face degree increase by one,
	\item lets the number of internal vertices decrease by one,
	\item and preserves the number of internal faces, as well as their left and right lengths, except for the marked face, whose left length decreases by one.
\end{itemize}
\end{thm}

\begin{proof}
We already know from Propositions~\ref{meinBbr} and~\ref{mvinBv} that~$\Phi$ and~$\Psi$ take their values in the proper sets~$\cBbr$ and~$\cBv$.

We need to see that, for $(\m,v)\in \cBv$, we have $\Psi(\Phi(\m,v))=(\m,v)$. For this, notice that the sliding path of $(\m,v)$ becomes the sliding path of $\Phi(\m,v)$ through the mapping~$\Phi$ and that the operations of sliding up and sliding down are clearly inverse one from another. For the same reason, $\Phi\circ\Psi$ is the identity on~$\cBbr$.

Finally, the stated parameter correspondences of the mapping~$\Phi$ are also direct consequences of the construction.
\end{proof}

\paragraph{Specializations.}
We now state the specializations to the three families under consideration. We define the following subsets of~$\cB$:
\begin{itemize}
	\item the set $\cT_{k,j}$ of bipolar oriented quasi-triangulations with~$k$ internal vertices and an external face of degree~$j$ (so that $T_{k,j}=|\cT_{k,j}|$) for any $k\geq 0$, $j\geq 2$\,;
	\item the set $\cB_{k,\ell,j}$ of bipolar oriented maps with~$k$ internal vertices, $\ell$ internal faces, and an external face of degree~$j$ (so that $B_{k,\ell,j}=|\cB_{k,\ell,j}|$) for any $k\geq 0$, $\ell\geq 1$, $j\geq 2$\,;
	\item the set $\cS_{k,j}$ of bipolar oriented maps with~$k$ internal vertices, an external face of degree~$j$, and such that all internal faces have right length~$2$ (so that $S_{k,j}=|\cS_{k,j}|$) for any $k\geq 0$, $j\geq 3$.
\end{itemize}
Similarly as we did above for~$\cB$, we define for any of these subsets~$\cX$ two variations:
\begin{itemize}
	\item the set~$\cXv$ of pairs $(\m,v)$ where $\m\in\cX$ and~$v$ is an internal vertex of~$\m$;
	\item the set~$\cXbr$ of pairs $(\m,e)$ where $\m\in\cX$ and~$e$ is a boundary-reaching right-internal edge.
\end{itemize}

We readily obtain the following from Theorem~\ref{thmsss}.

\begin{corol}\label{corsssbs}
The slit-side-sew correspondence specializes into bijections between:
\begin{itemize}
	\item the set $\cBv_{k,\ell,j}$ and the set $\cBbr_{k-1,\ell,j+1}$ for any $k\geq 1$, $\ell\geq 1$, $j\geq 2$\,;
	\item the set $\cSv_{k,j}$ and the set $\cSbr_{k-1,j+1}$ for any $k\geq 1$, $j\geq 3$.
\end{itemize}	
\end{corol}

\paragraph{Quasi-triangulations.}
Because of the drop of degree in the marked face (at the right of the distinguished vertex), we will need to perform an extra ``squeezing'' step for quasi-triangulations. As a result, the property that, via~$\Phi$, the distinguished edge is right-internal will no longer hold in this case; we rather define the set~$\cTbe$ (resp.\ $\cTbe_{k,j}$) as the set of pairs $(\m,e)$ where $\m\in\cT$ (resp.\ $\m\in\cT_{k,j}$) and~$e$ is a boundary-reaching edge of~$\m$.

We then amend the bijections~$\Phi$ and~$\Psi$ as follows. Let $(\m,v)\in\cTv_{k,j}$ and $(\tilde\m,\tilde e)\de\Phi(\m,v)$. We add an extra step \textbf{at the end}:

\begin{description}[style=nextline]
\item[\textcolor{stepcol}{D. Shrinking the degree 2-face}]
	The face~$\tilde f$ at the right of~$\tilde e$ is a degree 2-face. Let~$\hat e$ be the edge on the right of~$\tilde f$. Note that it has same origin and end as~$\tilde e$. We collapse~$\tilde f$ by sewing~$\hat e$ onto~$\tilde e$, and denote by~$\widehat\m$ the resulting map. We set $\smash{\widehat\Phi}(\m,v)\de(\widehat\m,\hat e\sew\tilde e)$.
\end{description}
Conversely, we take $(\widehat\m,\hat e)\in\cTbe_{k-1,j+1}$ and add an extra step \textbf{at the beginning}:
\begin{description}[style=nextline]
\item[\textcolor{stepcol}{0. Blowing the distinguished edge}]
	We replace the distinguished edge~$\hat e$ with an extra degree 2-face by doubling the edge, while keeping the orientation of the doubled edge. Among the two edges resulting from the doubling of~$\hat e$, we let~$\tilde e$ be the one having the extra degree 2-face to its right. We then let~$\tilde\m$ be the resulting map and set $\smash{\widehat\Psi}(\widehat\m,\hat e)\de \Psi(\tilde\m,\tilde e)$.
\end{description}

Since these two extra steps are clearly inverse one from another, we obtain the following.

\begin{corol}\label{corssst}
The slit-side-sew correspondence extends into bijections $\hat\Phi\colon \cTv \to \cTbe$ and $\hat\Psi\colon \cTbe \to \cTv$, inverse one from another, which, for $k\geq 1$, $j\geq 2$, specialize into bijections between~$\cTv_{k,j}$ and~$\cTbe_{k-1,j+1}$.
\end{corol}

\section{Boundary-reaching probabilities}\label{secbr}

\paragraph{Cardinalities.}
We come back to the prefactor~\ref{prefb} from the end of Section~\ref{secenum}. We add several new variations to a subset $\cX\subseteq\cB$: for any symbol $\xs\in\{\B\edge,\B\iedge,\B\redge\}$, we define the set~$\cXx$ of pairs $(\m,e)$ where $\m\in\cX$ and~$e$ is 
\begin{itemize}
	\item any edge of~$\m$ if $\xs=\B\edge$\,;
	\item an \emph{internal} edge of~$\m$, that is, an edge having internal faces on both sides, if $\xs=\B\iedge$\,;
	\item a right-internal edge of~$\m$ if $\xs=\B\redge$.
\end{itemize}
Furthermore, we define the subset~$\cXbx$ of pairs $(\m,e)\in\cXx$ where~$e$ is also boundary-reaching. With this notation, the enumeration at the end of Section~\ref{secenum} yields 
\begin{align*}
|\cTe_{k-1,j+1}|		&=\big(3k+2j-4\big)\,T_{k-1,j+1},\\
|\cBi_{k-1,\ell,j+1}|	&=\big(k+\ell-2\big)\,B_{k-1,\ell,j+1},\\
|\cSr_{k-1,j+1}|		&=\big(2k+j-3\big)\,S_{k-1,j+1},
\end{align*}
for the values of~$k$, $\ell$, $j$ where those are defined.

Note that, for any given map, the internal edges are also right-internal. On the contrary, the only edge that is right-internal but not internal is the root. Since the root is obviously never boundary-reaching, we actually have $\cXbr=\cXbi$. The reason why we consider internal edges rather than right-internal edges in the case of general bipolar oriented maps should become clear in what follows.

\paragraph{Boundary-reaching probabilities.}
We now introduce the \emph{boundary-reaching probabilities}, defined as the ratios
\begin{equation}
\tau_{k,j}		\de\frac{|\cTbe_{k,j}|}		{|\cTe_{k,j}|},\qquad
\beta_{k,\ell,j}	\de\frac{|\cBbi_{k,\ell,j}|}	{|\cBi_{k,\ell,j}|},\qquad
\sigma_{k,j}		\de\frac{|\cSbr_{k,j}|}		{|\cSr_{k,j}|}.
\end{equation}
In words, those are the probabilities that the distinguished edge~$e$ is boundary-reaching, for a uniformly chosen random pair $(\m,e)$ in a given class ($\cTe_{k,j}$, $\cBi_{k,\ell,j}$, or~$\cSr_{k,j}$). 

Recalling that $\cBbi_{k,\ell,j}=\cBbr_{k,\ell,j}$, Corollaries~\ref{corsssbs} and~\ref{corssst} then give a proof to Propositions~\ref{propt}, \ref{propb}, \ref{props}, provided that the following identities hold.

\begin{prop}\label{proppb}
We have  
\begin{align}
\tau_{k,j}		&=1-\frac{2}j	&\text{$k\geq 0$, $j\geq 2$,}				\label{eqratiot}\\ 
\beta_{k,\ell,j}	&=1-\frac{2}j	&\text{$k\geq 0$, $\ell\geq 1$, $j\geq 2$,}	\label{eqratiob}\\ 
\sigma_{k,j}		&=1-\frac{3}j	&\text{$k\geq 0$, $j\geq 3$.} 			\label{eqratios}
\end{align}
\end{prop}

These will be obtained as consequences of properties satisfied by orientations grouped into orbits of a rerooting operation. We will actually see that for the first two cases, the ratio formula holds in a much stronger sense, namely the probability that~$e$ is boundary-reaching is equal to $1-2/j$ for \emph{any fixed pair} $(\m,e)$, taking a uniformly random bipolar orientation of~$\m$, upon allowing the root edge to be any external edge. For the third case, the ratio formula holds in a similarly strong sense when translated to quasi-triangulations.

\subsection{Rerooting operator for bipolar oriented maps}

In this section, we focus on~\eqref{eqratiot} and~\eqref{eqratiob}. Since these obviously hold when $j=2$, we will restrict ourselves to the case $j\ge 3$.

\paragraph{Rerooting operator.}
For a planar map~$\m$ and a marked oriented edge $\vec \rho=(u,v)$ of~$\m$, a \emph{bipolar orientation of~$\m$ rooted at~$\vec \rho$} is an orientation of all the edges of~$\m$ in such a way that the resulting oriented map is a bipolar oriented map with root edge~$\vec \rho$. It is known~\cite{Lem67} that~$\m$ admits a bipolar orientation rooted at~$\vec \rho$ if and only if~$\m$ is 2-connected, that is, such that the deletion of any single vertex does not disconnect~$\m$.

Moreover, for another marked directed edge $\vec \rho\,'=(u',v')$ of~$\m$, it is known~\cite{FrOsRo95} that the bipolar orientations of~$\m$ rooted at~$\vec \rho$ are in bijection with the bipolar orientations of~$\m$ rooted at~$\vec \rho\,'$. Consequently, the number of bipolar orientations of~$\m$ does not depend on the choice of the marked directed edge. In the special case -- the one we will use in the present work -- where $u'=v$, the bijection, called \emph{rerooting operator}, works as follows: given a bipolar orientation of~$\m$ rooted at~$\vec \rho$, an edge is flipped (in the sense that its orientation is reverted) if and only if there does not exist a directed path from it to~$v'$. See Figure~\ref{reroot}.

\begin{figure}[ht!]
	\centering\includegraphics[width=6cm]{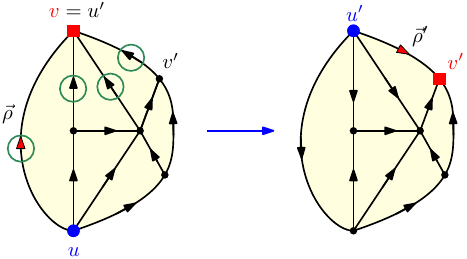}
	\caption{Rerooting operator in the case where the tail of the new root edge is the head of the former root edge. In this context, we represent North poles with red squares and South poles with blue disks. The 4 edges to be flipped are circled in green.}
	\label{reroot}
\end{figure}

\paragraph{Orbits.}
A \emph{rooted map} is a planar map with a marked oriented edge, called the \emph{root edge}, all the other edges not being oriented. As above, the \emph{external face} is the one on the left of the root edge, and is always drawn as the unbounded component of the plane.

For a 2-connected map~$\m$ with root edge~$\vec \rho$, and external face degree $j\geq 3$, let $\vec \rho=\vec \rho_0$, \dots, $\vec \rho_{j-1}$ be the oriented edges in clockwise order around the external face. We denote by~$\cB_i(\m)$ the set of bipolar oriented maps on~$\m$ rooted at $\vec \rho_i$, and set $\cB(\m)\de\bigcup_{i=0}^{j-1}\cB_i(\m)$.

The rerooting operator from the previous paragraph yields a bijection~$\sigma$ from~$\cB(\m)$ to~$\cB(\m)$, specializing into a bijection from~$\cB_i(\m)$ to~$\cB_{i+1}(\m)$ for every $0\leq i<j$. Note that the bijection~$\sigma$ amounts to flipping the edges that are not boundary-reaching, and take $\vec \rho_{i+1}$ as the new root edge.

An \emph{orbit} of $\cB(\m)$ is a cyclic sequence $\Orb$ of distinct elements in $\cB(\m)$ such that for any pair~$\bx$, $\bx'$ of successive elements of~$\Orb$, one has~$\bx'=\sigma(\bx)$. See Figure~\ref{orbit} for an example. Note that the length of any orbit has to be a multiple of~$j$, since the index~$i$ of the root edge goes up by~$1$ (modulo~$j$) from one orbit element to the next one.  

\begin{figure}[ht!]
	\centering\includegraphics[width=.95\linewidth]{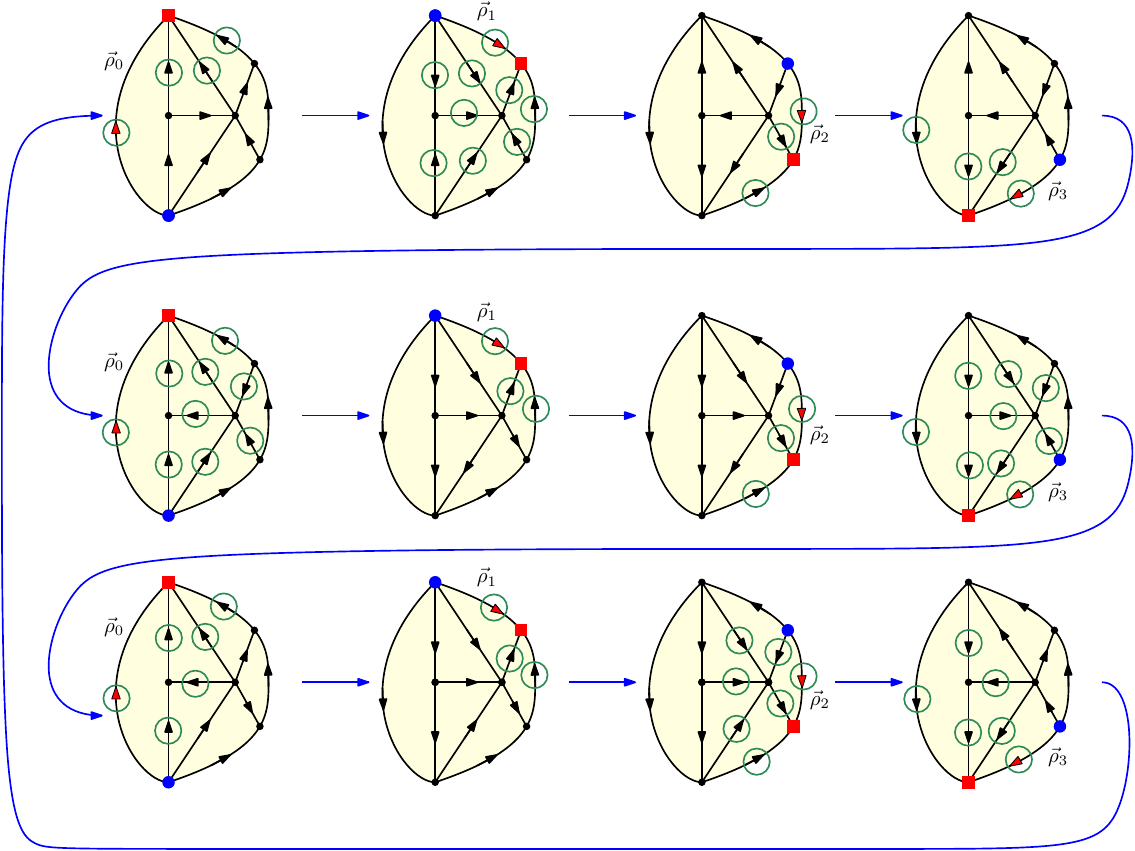}
	\caption{An orbit of length $12$ on bipolar oriented maps with external face degree $j=4$. At each step, the edges that are not boundary-reaching are circled in green; these are the edges to be flipped. One can note that for any given edge~$e$, there are exactly~6 elements of the orbit such that~$e$ is not boundary-reaching.}
	\label{orbit}
\end{figure}

\paragraph{Boundary-reaching submap.}
Given a bipolar oriented map $\bx\in\cB(\m)$, we define its \emph{boundary-reaching submap}~$\tilde{\bx}$ as the set of all boundary-reaching edges of~$\bx$. It is easy to see that it forms a connected submap that contains all the vertices of the boundary of~$\bx$ except the North pole. Since we assume that $j\geq 3$, the North pole has two neighbors on the boundary of~$\bx$: the South pole and another vertex~$v$. We define the \emph{separating path}~$P$ as the leftmost path that is incoming at~$v$ and starts from the South Pole, or, equivalently, the unique directed path from~$\rS$ to~$v$ such that any edge~$e$ on~$P$ is the leftmost incoming edge at~$e^+$. Note that all the edges of~$P$ are boundary-reaching; moreover by planarity, $P$ actually separates~$\tilde{\bx}$ from its complement in~$\bx$, in the sense that the boundary-reaching edges of~$\bx$ are exactly those on~$P$ or to its right. See Figure~\ref{figorbit_prop}.

One can then rephrase the rerooting operation in terms of these objects: to obtain~$\sigma(\bx)$ from~$\bx$, flip all edges of $\bx\setminus \tilde{\bx}$ and set the new root edge as the one succeeding the original one along the boundary in clockwise order, that is, the one from the original North pole to the tip of the separating path.

\begin{figure}[ht!]
	\centering\includegraphics[width=12cm]{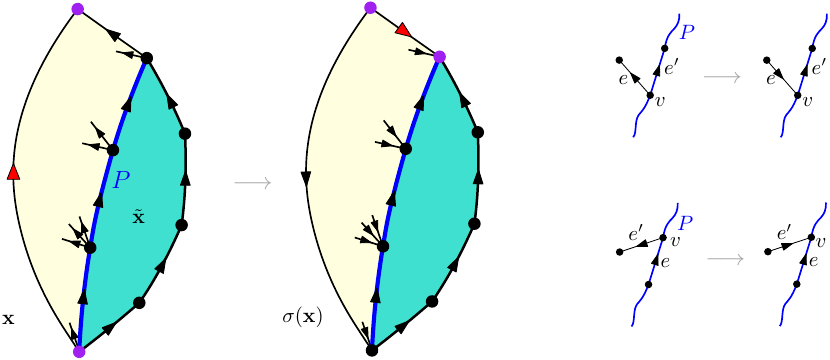}
	\caption{\textbf{Left.} Two successive elements~$\bx$, $\sigma(\bx)$ in an orbit; the flipped edges are those that are not boundary-reaching, that is, those outside of~$\tilde{\bx}$. The separating path~$P$ is in blue and the boundary-reaching submap~$\tilde{\bx}$ in turquoise. \textbf{Right.} For two consecutive edges~$e$, $e'$ around a vertex~$v$ (with the corner they delimit in an internal face), the two situations where only one of~$e$, $e'$ is flipped in the transition from~$\bx$ to~$\sigma(\bx)$.}
	\label{figorbit_prop}
\end{figure}

\paragraph{Orbit property.}
Let us now compute the average probability over a given orbit that a given edge is boundary-reaching.

\begin{lem}\label{lem:orbit}
Let~$\m$ be a rooted 2-connected map of external face degree $j\geq 3$. Let $\Orb$ be an orbit of $\cB(\m)$. Then, for any edge~$e$ of~$\m$, there is a proportion $2/j$ of elements in $\Orb$ for which~$e$ is not boundary-reaching.
\end{lem}

\begin{proof}
For a fixed orbit $\Orb$, the \emph{multiplicity} of an edge~$e$ of~$\m$ is the number of elements in~$\Orb$ such that~$e$ is not boundary-reaching, that is, is to be flipped in the transition to the next element of~$\Orb$. Letting $a\de |\Orb|/j$, we thus have to prove that every edge has multiplicity~$2a$. The property is very easy to check for edges incident to the external face. Indeed, for $0\le i<j$, the elements in $\Orb$ such that the (nonoriented) edge corresponding to~$\vec \rho_i$ is not boundary-reaching are exactly those where the root edge is either~$\vec \rho_{i-1}$ or~$\vec \rho_{i}$. 

\bigskip
We now show that the property propagates step by step to any edge by showing that two edges~$e$, $e'$ consecutive in clockwise order around a vertex~$v$ have the same multiplicity. This clearly allows to conclude the proof since, for every edge~$e$, one can find a sequence of edges starting with an external edge (for which the property is known), ending at~$e$, and where two successive edges in the sequence are consecutive around a vertex. We may assume that the corner~$c$ between~$e$ and~$e'$ is in an internal face since, otherwise, it means that~$e$ and~$e'$ are in the external face and we already know that they have same multiplicity. 

We introduce a bit of terminology. For any orientation of the edges of~$\m$, the corner~$c$ between~$e$ and~$e'$ is called \emph{extremal} if~$e$ and~$e'$ are either both incoming or both outgoing, and is called \emph{lateral} otherwise. We call \emph{special} the elements of $\Orb$ for which~$e$ and~$e'$ have a different status (boundary-reaching or not boundary-reaching). Let $\bx\in\Orb$; the following holds.
\begin{itemize}
	\item If~$\bx$ is not special, then~$e$ and~$e'$ both belong to the same submap~$\tilde\bx$ or $\bx\setminus \tilde{\bx}$. So, they are either both flipped or both left unchanged by the rerooting operation. As a result, $c$ has the same status (lateral or extremal) in~$\bx$ and in~$\sigma(\bx)$. 
	\item If~$\bx$ is special, then there are two possibilities, shown on the right of Figure~\ref{figorbit_prop}.
	\begin{description}
		\item[\textit{Top case.}] If~$e$ is not boundary-reaching (and thus~$e'$ is boundary-reaching) in~$\bx$, then it means that~$e$ is just on the left of~$P$ and~$e'$ is on~$P$, both~$e$ and~$e'$ going out of~$v$. Then~$c$ is extremal in~$\bx$ and lateral in~$\sigma(\bx)$.
		\item[\textit{Bottom case.}] If~$e$ is boundary-reaching (and thus~$e'$ is not boundary-reaching) in~$\bx$, then it means that~$e$ is on $P$ and~$e'$ is just on the left of~$P$ (possibly, $e$ being the last edge on the right boundary of~$\m$), with~$e$ incoming  and~$e'$ outgoing at $v$. Then~$c$ is lateral in~$\bx$ and extremal in~$\sigma(\bx)$. 
	\end{description}
\end{itemize}   

Monitoring the status of~$c$ along the orbit~$\Orb$, we see that the special elements in $\Orb$ alternate between the top case and the bottom case; the number of these special elements is thus evenly spread among the two cases. This entails that~$e$ and~$e'$ have the same multiplicity.    
\end{proof}

\begin{rem}\label{rem:homomesy}
Lemma~\ref{lem:orbit} is an instance of the \emph{homomesy} phenomenon~\cite{PrRo15}, which has been actively studied in the context of the rowmotion operator~\cite{Str15,DeHoPo23}.   
Namely, for a rooted 2-connected map~$\m$ of external degree~$j$, and for an edge~$e$ of~$\m$, Lemma~\ref{lem:orbit} exactly states that the indicator function for the event that~$e$ is not boundary-reaching is $\frac{2}{j}$-mesic on~$\cB(\m)$ under the rerooting operator.   

As reviewed in~\cite[Section~2.2]{PrRo15}, a well-known example where homomesy explains the prefactor in a counting formula is the cycle lemma for counting positive words -- words such that each non-empty prefix has more $1$'s than $0$'s -- in $\mathfrak{S}(1^b\,0^a)$ (for some fixed $b>a$). These are counted by $\frac{b-a}{a+b}\binom{a+b}{a}$, where the prefactor $\frac{b-a}{a+b}$ gives the proportion of positive words in each orbit under the shift operator. 
The situation for orbit sizes is rather simple in this case: they are of the form $(a+b)/d$, where $d\,|\,(b-a)\land(a+b)$.  In contrast, in our case, simulations show that the orbit sizes in $\cB(\m)$ vary a lot, and some are much larger than the number of edges of~$\m$.  
\end{rem}

\begin{rem}\label{rem:corresp}
For every external edge~$\rho$ of~$\m$, there is a clear $j$-to-$2$ correspondence between the elements of $\Orb$ and the elements of $\Orb$ such that~$\rho$ is not boundary-reaching: indeed, for~$j$ consecutive elements of $\Orb$ there are exactly two elements where~$\rho$ is not boundary-reaching (the one mentioned at the beginning of the proof of Lemma~\ref{lem:orbit}). Moreover, for any two edges~$e$, $e'$ of~$\m$ that are consecutive around a vertex~$v$, the proof ensures that there is a bijection $\iota_{e,e'}$ from $\Orb$ to $\Orb$ such that~$e$ is boundary-reaching in $\bx\in\Orb$ if and only if~$e'$ is boundary-reaching in $\iota_{e,e'}(\bx)$: if~$\bx$ is not special, we set $\iota_{e,e'}(\bx)\de\bx$, and otherwise, we let $\iota_{e,e'}(\bx)$ be the next special element along the orbit. By propagation, we thus obtain, for each edge~$e$ of~$\m$, a bijection~$\iota_e$ from $\Orb$ to $\Orb$ such that~$e$ is boundary-reaching in~$\bx$ iff an arbitrarily fixed external edge~$\rho$ is boundary-reaching in $\iota_e(\bx)$. Via this bijection we thus have a $j$-to-$2$ correspondence between the elements of $\Orb$ and the elements of $\Orb$ such that~$e$ is not boundary-reaching. 
\end{rem}

\begin{corol}\label{coro:good_proportion}
Let~$\m$ be a rooted 2-connected map of external face degree $j\geq 3$. Then, for any edge~$e$ of~$\m$, there is a proportion $2/j$ of elements in $\cB(\m)$ for which~$e$ is not boundary-reaching.
\end{corol}

\begin{proof}[Proof of Proposition~\ref{proppb}, Equations~\eqref{eqratiot} and~\eqref{eqratiob}]
Note that $j\,|\cTe_{k,j}|$ (resp.\ $j\,|\cTbe_{k,j}|$) is the number of triples $(\m,\bx,e)$ where~$\m$ is a bipolar oriented quasi-triangulation with~$k$ internal vertices and~$j$ external vertices, $\bx\in\cB(\m)$, and~$e$ is an edge (resp.\ a boundary-reaching edge) of~$\m$. Corollary~\ref{coro:good_proportion} then ensures that $j\,|\cTbe_{k,j}|=\big(1-\frac{2}{j}\big)\,j\,|\cTe_{k,j}|$, so that $\tau_{k,j}=1-2/j$, as desired. 

Similarly, $j\,|\cBi_{k,\ell,j}|$ (resp.\ $j\,|\cBbi_{k,\ell,j}|$) is the total number of triples $(\m,\bx,e)$ where~$\m$ is a rooted 2-connected map with~$k$ internal vertices, $\ell$ internal faces, and~$j$ external vertices, $\bx\in\cB(\m)$, and~$e$ is an internal edge (resp.\ a boundary-reaching internal edge) of~$\m$. Corollary~\ref{coro:good_proportion} then ensures that $j\,|\cBbi_{k,\ell,j}|=\big(1-\frac{2}{j}\big)\,j\,|\cBi_{k,\ell,j}|$, so that $\beta_{k,\ell,j}=1-2/j$, as wanted.
\end{proof}

\subsection{Rerooting operator for Schnyder woods}

We now turn to~\eqref{eqratios}. The overall strategy is the same as in the previous section. We will, however, work directly with Schnyder woods in a first step, considering the rerooting operator on so-called 3-orientations. In a second step, we will return back to bipolar orientations, right after Equation~\eqref{ratiocQLf}.

\paragraph{3-orientations and rerooting mapping.}
For a Schnyder wood on a triangulation~$\bt$, we convene that the root edge of~$\bt$ is the external edge $(\rho_2,\rho_3)$. Similarly to bipolar orientations, the number of Schnyder woods of a simple planar triangulation~$\bt$ does not depend on the choice of the root edge. In order to establish this claim, we use so-called \emph{3-orientations}, that is, orientations of the internal edges such that every internal vertex has outdegree~$3$ and the~$3$ external vertices~$\rho_1$, $\rho_2$, $\rho_3$ have outdegree~$0$.

It is known~\cite{Fel04} that, for each 3-orientation of~$\bt$, there is a unique Schnyder wood whose underlying orientation is the given 3-orientation. It is obtained by ``propagating the colors'' according to the local rule depicted on the right of Figure~\ref{schnyd}. More precisely, for any internal edge~$e$, one considers the so-called \emph{straight path}~$P_e$ of~$e$, which is the unique directed path starting at~$e$, ending at one of the external vertices~$\rho_i$ and having one outgoing edge on each side at every vertex it passes by; then the color assigned to~$e$ is that of~$\rho_i$ in the sense that $e\in\bt_i$. Consequently, once a root edge of~$\bt$ is fixed, its 3-orientations bijectively correspond to its Schnyder woods.

Let $(\rho_2,\rho_3)$ be a directed edge of~$\bt$, let~$\rho_1$ be the other vertex incident to the face on the left of $(\rho_2,\rho_3)$, and let~$z$ be the other vertex incident to the face on the left of $(\rho_1,\rho_3)$. See the left column in Figure~\ref{triangulation2}. There is a simple bijection between the 3-orientations of~$\bt$ rooted at $(\rho_2,\rho_3)$ and those rooted at $(\rho_3,z)$, which proceeds as follows. We revert the straight path~$P$ of the unique outgoing edge of~$z$ not leading to~$\rho_1$ or~$\rho_3$, we discard the orientations of $(z,\rho_1)$ and $(z,\rho_3)$, and we orient the edges $(\rho_2,\rho_3)$ and $(\rho_2,\rho_1)$ outward of~$\rho_2$ (note that~$P$ corresponds to the unique path of~$\bt_2$ from~$z$ to~$\rho_2$). See the middle column in Figure~\ref{triangulation2}. The claim follows.

\begin{figure}[ht!]
	\centering\includegraphics[width=.95\linewidth]{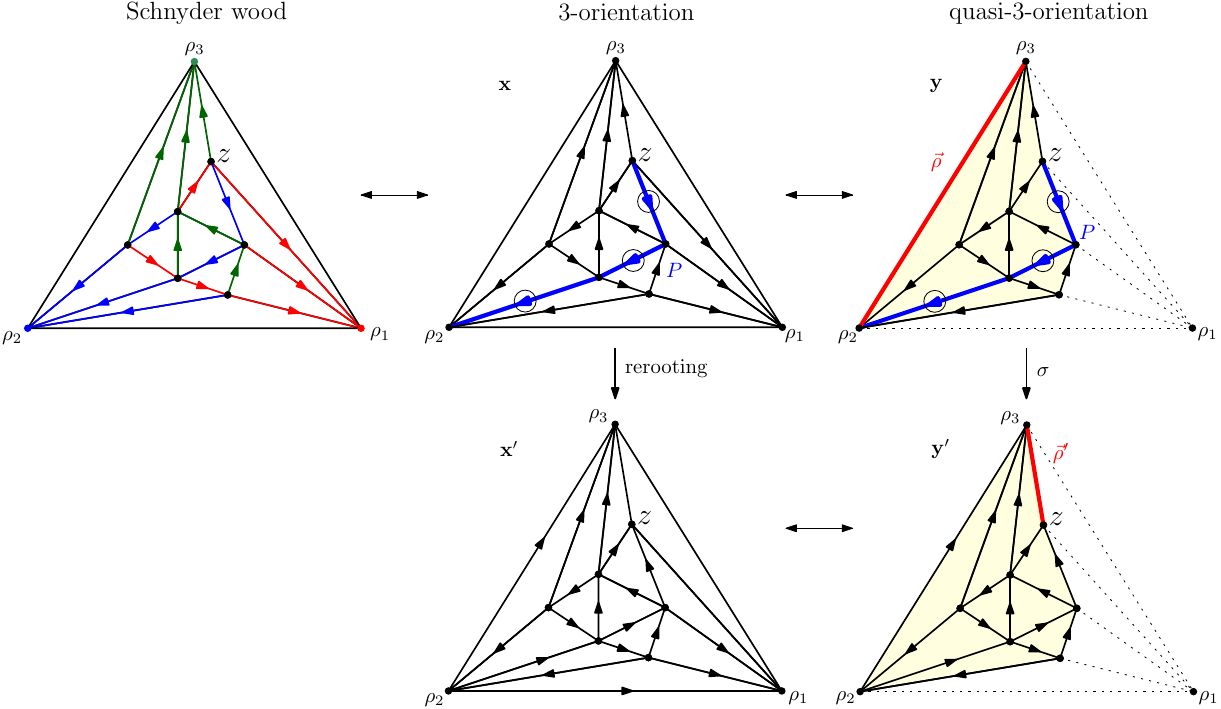}
	\caption{The rerooting mapping on a 3-orientation, and the induced rerooting operator~$\sigma$ on the associated quasi-3-orientation.}
	\label{triangulation2}
\end{figure}

\paragraph{Rerooting operator on quasi-3-orientations.}
We will actually adapt this rerooting mapping to the slightly different setting of quasi-triangulations. 
Given a simple quasi-triangulation~$\bq$ and~$\vec\rho$ one of its~$j$ external edges, a \emph{quasi-3-orientation} of~$\bq$ rooted at~$\vec\rho$ is an orientation of all the edges of~$\bq$ but~$\vec\rho$, such that:
\begin{itemize}
	\item the two extremities of~$\vec\rho$ have outdegree~$0$, 
	\item the other external vertices have outdegree~$2$, 
	\item the internal vertices have outdegree~$3$. 
\end{itemize}

For a simple triangulation~$\bt$ rooted at $(\rho_2,\rho_3)$, we obtain a quasi-triangulation~$\bq$ by deleting the other vertex incident to the face on the left of $(\rho_2,\rho_3)$ (denoted by~$\rho_1$ above), as well as its incident edges. If~$\bt$ has $k+j+1$ vertices then~$\bq$ has external face degree~$j$ and~$k$ internal vertices. Moreover, the 3-orientations of~$\bt$ clearly correspond to the quasi-3-orientations of~$\bq$.

Let~$\bt$ be a simple triangulation, let~$\vec\rho=(\rho_2,\rho_3)$ be an oriented edge of~$\bt$, let~$\bq$ be the quasi-triangulation corresponding to~$\bt$ rooted at~$\vec\rho$, and let~$\vec\rho\,'=(\rho_3,z)$. Let~$\bx$ be a 3-orientation of~$\bt$ rooted at~$\vec\rho$, and~$\bx'$ the 3-orientation obtained from the rerooting mapping described above. Let finally~$\by$ and~$\by'$ be the quasi-3-orientations of~$\bq$ corresponding respectively to~$\bx$ and~$\bx'$.

It is easy to directly describe how~$\by'$ is obtained from~$\by$. The \emph{separating path} for~$\by$ is the unique directed path~$P$ starting at~$z$ with the outgoing edge not leading to~$\rho_3$, ending at~$\rho_2$, and having a single outgoing edge on its right at every vertex it passes by. Then, in order to obtain~$\by'$ from~$\by$, we revert~$P$, discard the orientation of~$\vec\rho\,'$, and orient~$\vec\rho$ from~$\rho_2$ to~$\rho_3$. Note that the root edge has shifted by one unit in clockwise order along the external contour. We call \emph{rerooting operator} the mapping~$\sigma$ sending~$\by$ to~$\by'$. See the right column in Figure~\ref{triangulation2}.

\paragraph{Orbits.}
For a simple quasi-triangulation~$\bq$ with root edge~$\vec \rho$ and external face degree $j\geq 3$, let $\vec \rho=\vec \rho_0$, \dots, $\vec \rho_{j-1}$ be the oriented edges in clockwise order around the external face. For $0\leq i<j$, let $\cQ_i(\bq)$ be the set of quasi-3-orientations of~$\bq$ rooted at~$\vec \rho_i$, and set $\cQ(\bq)\de\bigcup_{i=0}^{j-1}\cQ_i(\bq)$. Similarly as in the previous section, the rerooting operator~$\sigma$ yields a bijection from~$\cQ_i(\bq)$ to~$\cQ_{i+1}(\bq)$, hence a bijection from~$\cQ(\bq)$ to~$\cQ(\bq)$.

\begin{figure}[ht!]
	\centering\includegraphics[width=.95\linewidth]{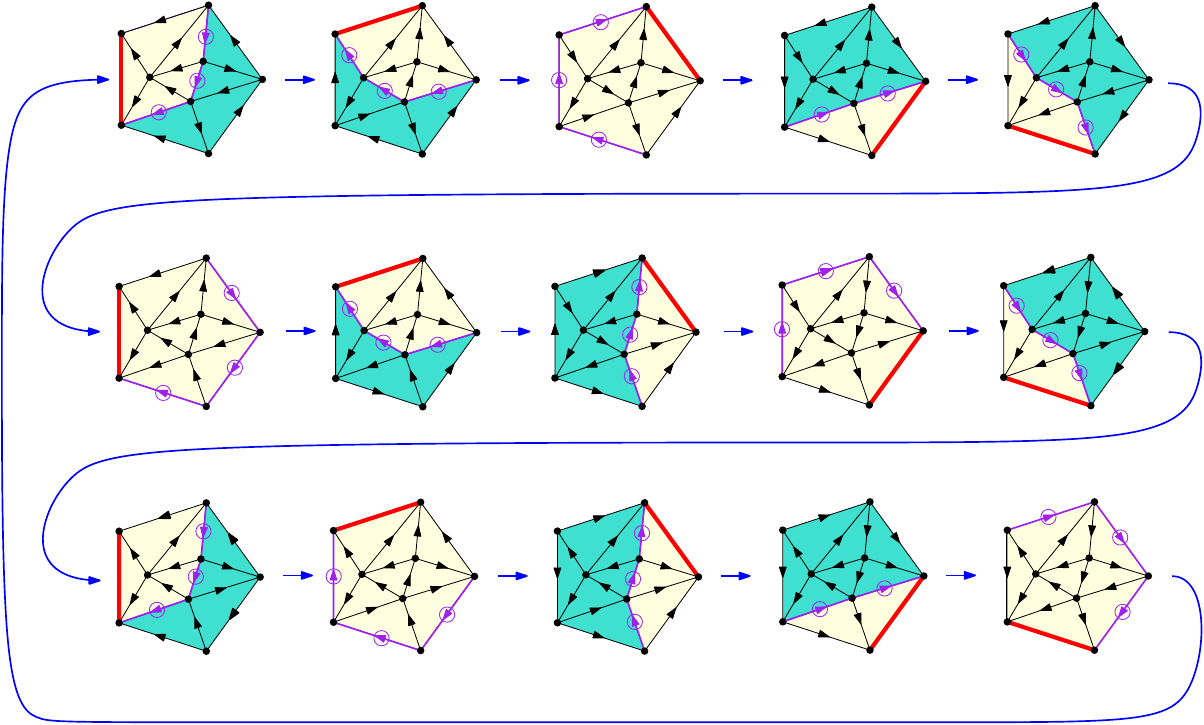}
	\caption{An orbit of length $15$ of quasi-3-oriented quasi-triangulations with external face degree $j=5$. For each element, the root edge is bold-red, the separating path is purple, and the right and left regions are respectively light yellow and turquoise. Each internal face appears~9 times in the right region (as ensured by the orbit property).}
	\label{orbit_schnyder}
\end{figure}

As above, an \emph{orbit} of~$\cQ(\bq)$ is a cyclic sequence $\Orb$ of distinct elements of~$\cQ(\bq)$ obtained by subsequent applications of~$\sigma$. Note that the length of any orbit is again a multiple of~$j$. See Figure~\ref{orbit_schnyder}. For $\bx\in\cQ(\bq)$, the separating path~$P$ of~$\bx$ yields a partition of the internal faces of~$\bq$ into two regions: the one on the left (resp.\ right) of~$P$ is called the \emph{left region} (resp.\ the \emph{right region}) for~$\bx$.

\paragraph{Orbit property.}
Similarly as above, the average probability over a given orbit that a given face is in the right region can be computed.

\begin{lem}\label{lem:orb_Schnyder}
Let~$\bq$ be a rooted simple quasi-triangulation of external face degree $j\geq 3$. Let $\Orb$ be an orbit of~$\cQ(\bq)$. Then, for every internal face~$f$ of~$\bq$, there is a proportion $3/j$ of elements in~$\Orb$ for which~$f$ is in the right region.
\end{lem}

\begin{proof}
For a fixed orbit $\Orb$, the \emph{multiplicity} of an internal face $f\in\bq$ is the number of elements in $\Orb$ such that~$f$ is in the right region. Letting $a=|\Orb|/j$ we thus have to prove that every internal face of~$\bq$ has multiplicity~$3a$. 

We proceed in two steps similarly as in Lemma~\ref{lem:orbit}, proving first the property for any internal face~$f$ adjacent to the external face and then propagating it to any internal face. So let~$f$ be incident to an external edge~$\vec\rho_i$ (and possibly a second external edge). Precisely, we will show that for any~$j$ consecutive elements~$\bx_0$, \ldots,~$\bx_{j-1}$ of~$\Orb$, with~$\bx_0$ rooted at~$\vec\rho_i$, there are exactly~$3$ elements such that~$f$ is in the right region; the situation is illustrated in Figure~\ref{border_face}, where~$e$ is the nonoriented edge corresponding to~$\vec\rho_i$.

\begin{figure}[ht!]
	\centering\includegraphics[width=.95\linewidth]{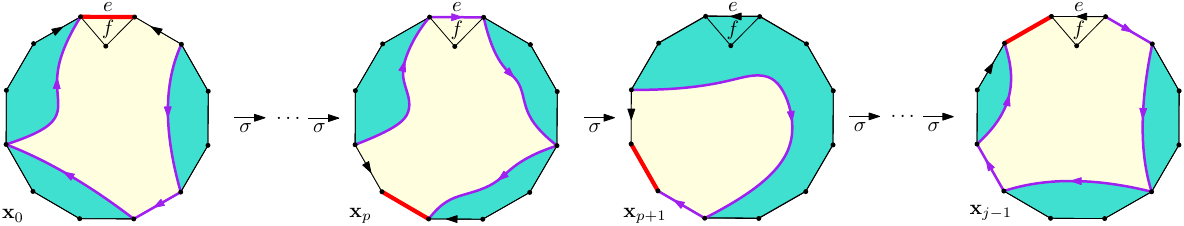}
	\caption{The situation in the proof of Lemma~\ref{lem:orb_Schnyder} for an internal face~$f$ incident to an external edge~$e$. For each depicted element, the root edge is bold-red, the separating path is purple, and the right and left regions are respectively light yellow and turquoise.}
	\label{border_face}
\end{figure}

Clearly,~$f$ is in the right region in~$\bx_0$ and in~$\bx_{j-1}$. Note that~$e$ is directed clockwise around the outer contour in~$\bx_1$, and counterclockwise in~$\bx_{j-1}$. Hence, there is a smallest $p\in\{1,\ldots,j-2\}$ such that~$e$ is clockwise in~$\bx_{p}$ and counterclockwise in~$\bx_{p+1}$. Since~$e$ is reverted when applying~$\sigma$ to~$\bx_{p}$, this means that~$e$ is on the separating path for~$\bx_p$. And necessarily, $f$ is on its right, hence in the right region.

Once~$e$ becomes counterclockwise (starting in~$\bx_{p+1}$), it has no chance of belonging to the separating path (and thus being reverted) before~$e$ becomes the root edge again. This comes from planarity and the fact that the separating path is a simple path. Thus~$\bx_p$ is the only element among~$\bx_1$, \ldots, $\bx_{j-2}$ for which~$e$ is on the separating path.

Note finally that, in each element of~$\Orb$ not rooted either at~$\vec\rho_i$ or at~$\vec\rho_{i-1}$, the only possibility for~$f$ to be in the right region is that~$e$ is on the separating path with~$f$ on its right. Among~$\bx_1$, \ldots, $\bx_{j-2}$, this only occurs for~$\bx_p$. Hence, among~$\bx_0$, \ldots, $\bx_{j-1}$, there are exactly~$3$ elements -- $\bx_0$, $\bx_p$, $\bx_{j-1}$ -- for which~$f$ is in the right region, as desired. 

\bigskip
Next, we show that two internal faces~$f$, $f'$ sharing an edge~$e$ have the same multiplicity. We call \emph{special} of type~$\mathrm{L/R}$ (resp.~$\mathrm{R/L}$) the elements of $\Orb$ for which~$f$ is in the left region while~$f'$ is in the right region (resp.~$f$ is in the right region while~$f'$ is in the left region), that is, $e$ is on the separating path, with~$f$ on the left (resp.\ on the right). Hence, for a special element in~$\Orb$, the next special element along the orbit is of the other type (these are the elements where~$e$ is flipped). This ensures that~$f$, $f'$ have same multiplicity. 

We conclude that any internal face~$f$ has multiplicity~$3a$, since there exists a sequence~$f_0$, \ldots, $f_k$ of internal faces ending at~$f$, such that~$f_0$ is incident to an external edge, and~$f_i$, $f_{i+1}$ share an edge for each $i\in\{0,\ldots,k-1\}$. 
\end{proof}
\begin{rem}
Similarly to Remark~\ref{rem:homomesy}, Lemma~\ref{lem:orb_Schnyder} is an instance of  homomesy: the indicator function that~$f$ is in the right region is $\frac{3}{j}$-mesic on~$\cQ(\bq)$ under the rerooting operator. 

And, similarly to Remark~\ref{rem:corresp}, for each internal face~$f$, one can design a $j$-to-$3$ correspondence between the elements of~$\Orb$ and those where~$f$ is in the right region. The correspondence is easy to describe for~$f$ incident to an external edge, and can then be propagated from face to (adjacent) face.  
\end{rem}

\begin{corol}\label{coro:good_proportion_Schnyder}
Let~$\bq$ be a rooted quasi-triangulation of external face degree $j\geq 3$. Then, for any internal face~$f$ of~$\bq$, there is a proportion $3/j$ of elements in $\cQ(\bq)$ for which~$f$ is in the right region.
\end{corol}

\begin{proof}[Proof of Proposition~\ref{proppb} Equation~\eqref{eqratios}]
Let $\cQ_{k,j}$ be the set of quasi-3-oriented quasi-triangulations with external face degree~$j$ having~$k$ internal vertices, and let $\cQf_{k,j}$ (resp.\ $\cQL_{k,j}$) be the set of pairs $(\bx,f)$ where $\bx\in\cQ_{k,j}$ and~$f$ is an internal face (resp.\ an internal face in the left region for~$\bx$). 
Then $j\,|\cQf_{k,j}|$ (resp.\ $j\,|\cQL_{k,j}|$) is the total number of triples $(\bq,\bx,f)$ where~$\bq$ 
is a rooted simple quasi-triangulation with~$k$ internal vertices and~$j$ external vertices, $\bx\in\cQ(\bq)$, 
and~$f$ is an internal face (resp.\ an internal face in the left region for~$\bx$). 
Corollary~\ref{coro:good_proportion_Schnyder} then ensures that
\begin{equation}\label{ratiocQLf}
j\,|\cQL_{k,j}|=\Big(1-\frac{3}{j}\Big)j\,|\cQf_{k,j}|\,.
\end{equation}

It remains to come back to bipolar oriented maps through the bijections we used. First of all, recall that~$\cS_{k,j}$ corresponds to~$\cQ_{k,j}$ via Schnyder woods, that is, through the bijection depicted in Figure~\ref{schnydbij} and then the ones in the top row of Figure~\ref{triangulation2}. Next, for an element $\bq\in\cQ_{k,j}$ corresponding to $\m\in\cS_{k,j}$, the following holds.

\begin{figure}[ht!]
	\centering\includegraphics[width=.9\linewidth]{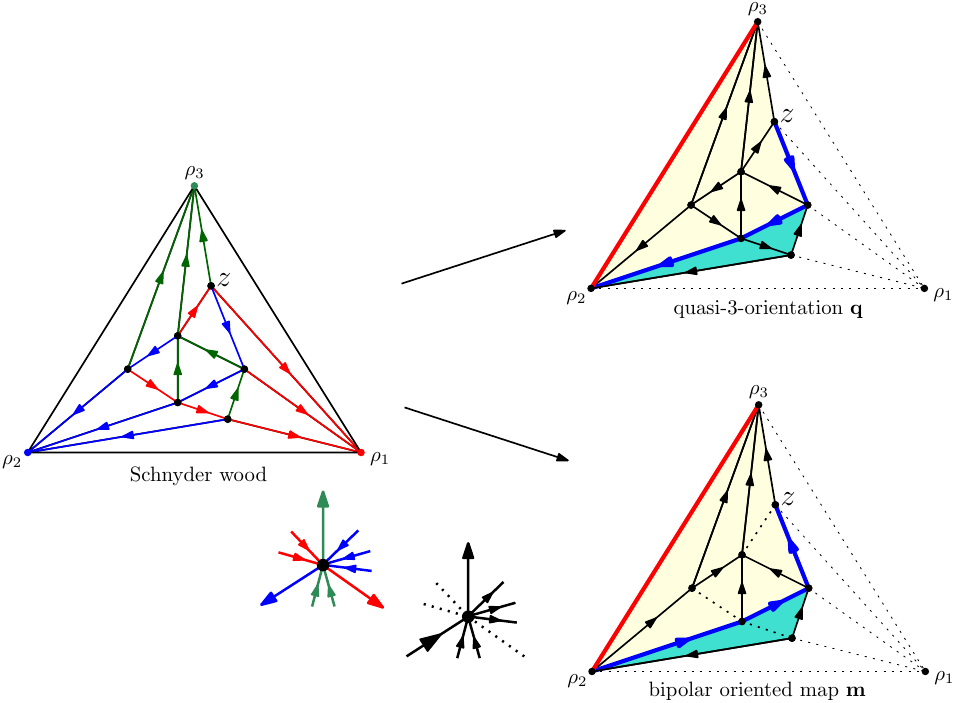}
	\caption{Illustration of the bijective correspondence between $\cQ_{k,j}$ (top right) and $\cS_{k,j}$ (bottom right) via Schnyder woods. For $\bq\in\cQ_{k,j}$ and the associated $\m\in\cS_{k,j}$, the separating paths coincide (in reverse directions). Hence, in the 1-to-1 correspondence between internal faces of~$\bq$ and right-internal edges of~$\m$, the internal faces in the left region correspond to the right-internal edges that are boundary-reaching.}
	\label{triangulation3}
\end{figure}

\begin{itemize}
	\item Each internal face~$f$ of~$\bq$ bijectively corresponds to a right-internal edge~$e$ of~$\m$: the face~$f$ is the triangular face to the right of~$e$ upon superimposing~$\bq$ and~$\m$. See Figure~\ref{schnydbij}. 
	\item The separating paths of~$\bq$ and of~$\m$ coincide, with opposite edge directions. Indeed, the outgoing blue edge of an internal vertex $v$ in the Schnyder wood corresponds to the leftmost incoming blue edge of $v$ in the bipolar oriented map. See Figure~\ref{triangulation3}.
	\item Hence, internal faces in the left region for~$\bq$ correspond to right-internal edges that are boundary-reaching for~$\m$. 
\end{itemize}

From these observations, the bijection between $\cQ_{k,j}$ and $\cS_{k,j}$ yields a bijection between  $\cQf_{k,j}$ and $\cSr_{k,j}$, which specializes into a bijection between $\cQL_{k,j}$ and $\cSbr_{k,j}$. As a result,
\[
\sigma_{k,j}=\frac{|\cSbr_{k,j}|}{|\cSr_{k,j}|}=\frac{|\cQL_{k,j}|}{|\cQf_{k,j}|}=1-\frac{3}{j},
\]
where the last equality is given by~\eqref{ratiocQLf}.
\end{proof}

\section{Base cases of the counting formulas}\label{secbase}

Since our bijections provide proofs to~\eqref{eqt}, \eqref{eqb}, \eqref{eqs}, we can recover~\eqref{nbt}, \eqref{nbb}, \eqref{nbs}, using the following base cases. Clearly, $T_{0,2}=1$.

\begin{prop}
For $j\geq 3$, the sets $\cT_{0,j}$ and $\cS_{0,j}$ are in bijection with rooted binary trees having $j-1$ leaves, so that
\[
T_{0,j}=S_{0,j}=\Cat_{j-2}=\frac{(2j-4)!}{(j-1)!(j-2)!}.
\]		
For $\ell\geq 1$, $j\geq 2$, the set $\cB_{0,\ell,j}$ is in bijection with rooted plane trees having $\ell$ nodes and $j-1$ leaves. These are counted by Narayana numbers:
\[
B_{0,\ell,j}=\Nar_{\ell,j-1}=\frac{(\ell+j-2)!(\ell+j-3)!}{\ell!(\ell-1)!(j-1)!(j-2)!}.
\]
\end{prop}

\begin{proof}
We call \emph{dissection} a rooted planar 2-connected map~$\m$ where all the vertices are incident to the external face. We claim that any dissection admits a unique bipolar orientation. Indeed, if we order the vertices around the external contour in counterclockwise order as~$v_0$, \dots, $v_{j-1}$ with~$v_0$ the origin of the root edge, then the condition for the right external boundary of bipolar oriented maps ensures that every edge $\{v_p,v_{p+1}\}$ has to be directed from~$v_p$ to~$v_{p+1}$ for $0\leq p\leq j-2$. And the acyclicity implies that every internal edge $\{v_p,v_q\}$, with $p<q$, has to be directed from~$v_p$ to~$v_q$. Hence, dissections are in bijection with bipolar oriented maps with no internal vertices.
 
In particular $\cT_{0,j}$ is in bijection with dissections of external face degree~$j$ and internal faces of degree $3$, and $\cB_{0,\ell,j}$ is in bijection with dissections with external face degree~$j$ and $\ell$ internal faces. By duality, these respectively correspond to rooted binary trees with $j-1$ leaves, and to rooted plane trees with~$\ell$ nodes and $j-1$ leaves, which are respectively counted by $\Cat_{j-2}$ and by~$\Nar_{\ell,j-1}$. 

\bigskip
In order to deal with $\cS_{0,j}$, it is convenient to observe that a bipolar oriented map has no internal vertices if and only if every internal face has left length~$1$. Indeed, if there is an internal vertex~$v$, then the internal face~$f$ on the right of the rightmost outgoing edge of~$v$ has its left lateral path passing by~$v$, hence of length at least~$2$; and if an internal face~$f$ has its left lateral path of length at least~$2$ then this path has to visit a vertex~$v$ having~$f$ on its right, so that~$v$ is necessarily an internal vertex.

As a consequence, for bipolar oriented maps with no internal vertex, having all internal faces of degree~$3$ is the same as having all internal faces of right length~$2$. Hence, $\cT_{0,j}=\cS_{0,j}$.
\end{proof}

\section{Concluding remarks}\label{sec:conclusion}

\subsection{Growth bijections for the base cases}

\paragraph{Catalan.}
Note that R\'emy's  bijection~\cite{Rem85}, which interprets the identity~\eqref{remeq}, yields a growth bijection for maps in~$\cT_{0,j}$ through the bijection via dissections from the previous section. Hence, there is a complete growth process for maps in $\cT_{k,j}$: use R\'emy's bijection to grow from $\varnothing$ to $\cT_{0,k+j}$, then use our bijection (Corollary~\ref{corssst}) $k$ times to grow from $\cT_{0,k+j}$ to $\cT_{k,j}$. The same holds for $\cS_{k,j}$.

\paragraph{Narayana.}
Regarding $\cB_{0,\ell,j}$, we have not been able to find in the literature a growth bijection for structures counted by Narayana numbers. Recall that $\Nar_{a,b}$ counts rooted plane trees with~$a$ nodes and~$b$ leaves, or, alternatively, rooted binary trees with~$a$ left and~$b$ right leaves.

For the sake of completeness, we briefly present such a bijection here. One easily obtains from the formula for $\Nar_{a,b}$ the following simple identity
\begin{equation}\label{eqNar}
(a+1)\,a\ \Nar_{a+1,b}=(b+1)\,b\ \Nar_{a,b+1},
\end{equation}
which, together with the base case $\Nar_{1,b}=1$ yields the formula for $\Nar_{a,b}$. Our bijection for~\eqref{eqNar} relies on cutting/joining operations similarly to~\cite{BeMo14,Mar22ckert}. We find it more convenient to describe it on rooted binary trees counted by left and right leaves (though it can also be described for rooted plane trees counted by nodes and leaves). 

For a rooted binary tree, we call \emph{inner edges} those connecting two nodes, and \emph{leaf edges} those incident to a leaf. An edge is called \emph{left} or \emph{right} according to whether it connects a parent to its left child or to its right child. In this terminology, $\Nar_{a,b}$ counts rooted binary trees with $a$ left leaf edges and $b$ right leaf edges. Such trees also have $a-1$ right inner edges and $b-1$ left inner edges. Hence, if we let $\cF_{a,b}$ (resp.\ $\cG_{a,b}$) be the set of rooted binary trees with $a$ left leaf edges, $b$ right leaf edges, and having a marked left leaf edge and a marked right inner edge (resp.\ a marked right leaf edge and a marked left inner edge), then~\eqref{eqNar} reads $|\cF_{a+1,b}| = |\cG_{a,b+1}|$.

The bijection between $\cF_{a+1,b}$ and $\cG_{a,b+1}$ is actually the specialization of an involution~$\chi$ on the set of rooted binary trees with a marked leaf edge of some type (left or right) and a marked inner edge of the other type. It is illustrated in Figure~\ref{Naraya} and proceeds as follows. We cut the (red) marked inner edge and attach to the (blue) marked leaf edge the cut tree that does not contain the marked leaf edge. We keep the markings in the process.

\begin{figure}[ht!]
	\centering\includegraphics[width=.9\linewidth]{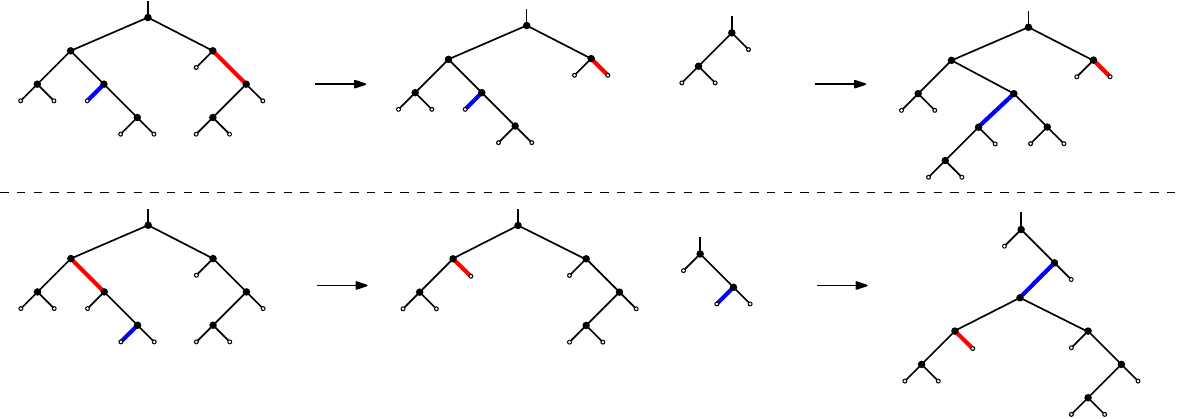}
	\caption{The involution~$\chi$ interpreting~\eqref{eqNar}. In the intermediate state, the two marked leaf edges may be in the same tree (top row), or in the 2 trees (bottom row). In this example, the original tree has~$5$ left leaf edges and~$4$ right leaf edges. Observe that the final tree has one less left leaf edges and one more right leaf edges, and that the types of the markings are reversed.}
	\label{Naraya}
\end{figure}

Observe that, in the intermediate state, we obtain a pair of rooted binary trees with a marked leaf edge of each types. In the end, the number of leaf edges of the type of the marked inner edge increases by one (in the cutting step), while the number of leaf edges of the type of the marked leaf edge decreases by one (in the attachment step). All in all, $\chi$ specializes into inverse bijections between $\cF_{a+1,b}$ and~$\cG_{a,b+1}$.

\bigskip
This provides the following complete growth process for $\cB_{k,\ell,j}$.
\begin{enumerate}
	\item Start from right comb binary tree having~$1$ right leaf and $k+\ell+j-2$ left leaves.
	\item Apply $k+j-2$ times the mapping~$\chi$ (each time increasing the number of right leaves) and obtain a binary tree having $k+j-1$ right leaves and~$\ell$ left leaves.
	\item Take the corresponding element in~$\cB_{0,\ell,k+j}$ through the classical bijections.
	\item Grow from $\cB_{0,\ell,k+j}$ to $\cB_{k,\ell,j}$ using our bijection~$k$ times (Corollary~\ref{corsssbs}).
\end{enumerate}

\subsection{Tableau interpretation of the identities}\label{sec:tableau}
\newcommand{\sdots}{\raisebox{-.09em}{$\cdot\hspace{-0.3em}\cdot\hspace{-0.3em}\cdot$}}

\paragraph{Hook formulas.}
Recall that the number of standard Young tableaux of given shape~$\lambda$ is given by the so-called \emph{hook length formula} (see e.g.\ \cite[Corollary~7.21.6]{Sta99}):
\begin{equation}\label{hlf}
\frac{n!}{\prod_{x\in\lambda} h_x},
\end{equation}
where~$n$ is the number of cells in~$\lambda$ and~$h_x$ is the \emph{hook length} of the cell~$x$, that is, the number of cells to its right in the same row or below it in the same column.

Recall also that the number of semistandard Young tableaux of shape~$\lambda$ whose entries are bounded by some~$K$ is given by the so-called \emph{hook content formula} (see e.g.\ \cite[Theorem~7.21.2]{Sta99}):
\begin{equation}\label{hcf}
\prod_{x\in\lambda} \frac{K+c_x}{h_x},
\end{equation}
where $c_x\de b-a$ is the \emph{content} of the cell~$x=(a,b)$.

\paragraph{Quasi-triangulations.}
It is known~\cite{Bor17,KeMiShWi19} that $\cT_{k,j}$ is in bijection with the set $\cY_{k,j}$ of standard Young tableaux of shape 
\[
\lambda_{k,j}=(k+j-2,k+j-2,k) =
\begin{tabular}{l}
$\overbrace{\hspace{5.92em}}^{k+j-2}$\\[-.5ex]
\begin{ytableau}
~ & ~ & \none[\sdots] & ~ & ~ & \none[\sdots] & ~ \\
~ & ~ & \none[\sdots] & ~ & ~ & \none[\sdots] & ~ \\
~ & ~ & \none[\sdots] & ~       
\end{ytableau}\\[-1.5ex]
$\underbrace{\hspace{3.42em}}_{k}$
\end{tabular},
\]
which directly yields~\eqref{nbt} via the hook length formula~\eqref{hlf}.

Identity~\eqref{eqt} is also easy to see. Indeed, $\lambda_{k-1,j+1}$ is $\lambda_{k,j}$ with the last entry in the third row deleted. The deletion lets all entries preserve their hook length except those in the third row and $k$-th column, where the hook length decreases by~$1$. Hence, setting $n=3k+2j-4$, Equation~\eqref{eqt} follows from the hook length formula~\eqref{hlf}:
\[
\frac{|\cY_{k,j}|}{|\cY_{k-1,j+1}|}=n\,\frac{\prod_{x\in\lambda_{k-1,j+1}}h_x}{\prod_{x\in\lambda_{k,j}}h_x}=n\frac1{k}\frac{j-1}{j+1}\,.
\]  

From this fact, it is also possible to derive an alternative bijective proof of~\eqref{eqt} via tableaux. Indeed, using the notation~$\wh{a}$ for the integer set $\{1,\dots,a\}$, the bijective proof of the hook length formula~\cite{NoPaSt97}  yields a bijection
\[
\bigg(\prod_{x\in\lambda_{k,j}}\wh{h_x}\bigg)\times\cY_{k,j}\simeq \wh{n}\times\bigg(\prod_{x\in\lambda_{k-1,j+1}}\wh{h_x}\bigg)\times\cY_{k-1,j+1}.
\]
Letting~$E$ be the multiset of integers formed by the hook lengths of the entries in $\lambda_{k,j}$, except for the first entry in the third row and the first entry in the $k$-th column, we thus have
\[
\Big(\prod_{a\in E}\wh{a}\Big)\times\wh{j+1}\times\wh{k}\times\cY_{k,j}\simeq \Big(\prod_{a\in E}\wh{a}\Big)\times\wh{j-1}\times\wh{n}\times\cY_{k-1,j+1}.
\]
We set $N\de\prod_{a\in E}a$. We thus obtain an $N$-to-$N$ correspondence between 
$\wh{j+1}\times\wh{k}\times\cY_{k,j}$ and $\wh{j-1}\times\wh{n}\times\cY_{k-1,j+1}$. Hence we also have an $N$-to-$N$ correspondence between
$\wh{j+1}\times\wh{k}\times\cT_{k,j}$ and $\wh{j-1}\times\wh{n}\times\cT_{k-1,j+1}$. Via Hall's marriage theorem, a 1-to-1 correspondence can be extracted from the $N$-to-$N$ correspondence, but without an explicit description (this kind of argument was previously used in bijective constructions, e.g.\ in~\cite[Theorem~5]{ChFeFu13}). 

\begin{rem}
As the bijective proof of the hook length formula -- and similarly that of the hook content formula -- proceeds via jeu-de-taquin operations, it seems unlikely that our bijective correspondence between $\wh{j+1}\times\wh{k}\times\cT_{k,j}$ and $\wh{j-1}\times\wh{n}\times\cT_{k-1,j+1}$ (see Remark~\ref{rem:corresp}) would have a simple link with the one between $\wh{j+1}\times\wh{k}\times\cY_{k,j}$ and $\wh{j-1}\times\wh{n}\times\cY_{k-1,j+1}$, if we relate the two correspondences via the known bijections~\cite{Bor17,KeMiShWi19} between $\cT_{k,j}$ and~$\cY_{k,j}$.
\end{rem}

\paragraph{General bipolar oriented maps.}
Regarding bipolar oriented maps counted by vertices and faces, it is known~\cite{FuPoSc09,AlPo15} that $\cB_{k,j,\ell}$ is in bijection with noncrossing triples of lattice walks with steps in $\{\mathrm{E},\mathrm{N}\}$, starting at the origin and ending at $(k+j-2,\ell-1)$, such that the upper walk ends with $NE^{j-2}$ for $\ell\geq 2$, and is equal to~$E^{j-2}$ for $\ell=1$. These are classically in bijection with the set $\cZ_{k,\ell,j}$ of semistandard Young tableaux on~$\lambda_{k,j}$ with entries bounded by $\ell+1$; Equation~\eqref{nbb} is thus the hook content formula~\eqref{hcf}. Furthermore, comparing~$\lambda_{k,j}$ with $\lambda_{k-1,j+1}$, the hook lengths differ as above; regarding contents, they are all the same except for the unique cell of $\lambda_{k,j}\setminus\lambda_{k-1,j+1}$, whose content is $k-3$. Letting $m=k+\ell-2$, Equation~\eqref{eqb} thus follows from the hook content formula~\eqref{hcf}:
\[
\frac{|\cZ_{k,\ell,j}|}{|\cZ_{k-1,\ell,j+1}|}= m\,\frac{\prod_{x\in\lambda_{k-1,j+1}}h_x}{\prod_{x\in\lambda_{k,j}}h_x}=m\frac1{k}\frac{j-1}{j+1}\,.
\] 
Then the bijective proof~\cite{Kra99} of the hook content formula yields an $N$-to-$N$ correspondence between $\wh{j+1}\times\wh{k}\times\cZ_{k,\ell,j}$ and $\wh{j-1}\times\wh{m}\times\cZ_{k-1,\ell,j+1}$, hence an $N$-to-$N$ correspondence between $\wh{j+1}\times\wh{k}\times\cB_{k,\ell,j}$ and $\wh{j-1}\times\wh{m}\times\cB_{k-1,\ell,j+1}$, from which a (non-explicit) 1-to-1 correspondence can be extracted via Hall's marriage theorem as above. 


\paragraph{Schnyder woods.}
Finally, regarding Schnyder woods, the bijection in~\cite{BeBo09} ensures that $S_{k,j}$ is the number of noncrossing pairs of lattice walks with steps in $\{\mathrm{E},\mathrm{N}\}$, starting at the origin and ending at $(k+j-2,k+j-2)$, staying above the diagonal, and such that the upper walk ends with $NE^{j-2}$. A simple application of the Lindstr\"om--Gessel--Viennot lemma then gives~\eqref{nbs}. These correspond to semistandard Young tableaux of shape
\[
(k+j-2,k) =
\begin{tabular}{l}
$\overbrace{\hspace{5.92em}}^{k+j-2}$\\[-.5ex]
\begin{ytableau}
~ & ~ & \none[\sdots] & ~ & ~ & \none[\sdots] & ~ \\
~ & ~ & \none[\sdots] & ~       
\end{ytableau}\\[-1.5ex]
$\underbrace{\hspace{3.42em}}_{k}$
\end{tabular},
\]
with entries bounded by $k+j-2$ and such that the entries in column~$i$ are at least~$i$. It would be interesting to have a bijective proof of~\eqref{eqs} via such tableaux as above. Alternatively, one could also expect a bijective proof via the symplectic tableaux that are obtained by a slightly different correspondence, as described in~\cite[Section~4]{KrGuVi00}.

\subsection{Random generation}\label{sec:random}

For the sake of completeness, we now discuss random generators obtained by our slit-slide-sew bijections. As we will see, the complexity, while polynomial, is quite high due to rejection induced by the boundary-reaching constraint. We discuss in details the case of $\cT_{k,j}$. 
Let $\Gamma T_{k,j}$ be the random sampler on~$\cT_{k,j}$ specified as follows.
\begin{itemize}
\item
If $k=0$, draw a uniformly chosen random binary tree with $j-2$ nodes (using e.g.\ R{\'e}my's procedure) and return the oriented map in~$\cT_{0,j}$ associated with it. 
\item
If $k\geq 1$, repeat:
\begin{itemize}
	\item $\tilde\m\leftarrow\Gamma T_{k-1,j+1}$
	\item $e\leftarrow$ uniformly chosen random edge in $\tilde\m$
\end{itemize}
until~$e$  is boundary-reaching. Open~$e$ into a face~$f$ of degree~$2$, and let~$\tilde e$ be the representative of~$e$ having~$f$ on its right. Let $(\m,v)=\Psi(\tilde\m,\tilde e)$. Return~$\m$. 
\end{itemize}
It is easily checked by induction on~$k$ that $\Gamma T_{k,j}$ is a uniform random sampler on~$\cT_{k,j}$. Let $\Lambda T_{k,j}$ be its expected complexity. Since the cost of computing the rightmost path from a given vertex in $\tilde\m\in\cT_{k-1,j+1}$ is at most $k-1$, we have, for $k\geq 1$ and $j\geq 2$,
\[
\Lambda T_{k,j}\leq \Lambda T_{k-1,j+1}+(k-1)+\frac{2}{j+1}\Lambda T_{k,j}.
\]
Hence,
\begin{align*}
\Lambda T_{k,j}&\leq \frac{j+1}{j-1}\Big(k-1+\Lambda T_{k-1,j+1}\Big)\,,\\
\shortintertext{so that, after iterating,}
\Lambda T_{k,j}&\leq \sum_{m=1}^k \frac{(j+m-1)(j+m)}{(j-1)j}(k-m) + \frac{(j+k-1)(j+k)}{(j-1)j}\Lambda T_{0,j+k}\\
\shortintertext{hence, bounding each summand by $\dfrac{(j+k-1)(j+k)}{(j-1)j}k$, we obtain}
\Lambda T_{k,j}&\leq  \frac{(j+k-1)(j+k)}{(j-1)j}\Big(k^2 + \Lambda T_{0,k+j}\Big)\,.
\end{align*}
Finally, $\Lambda T_{k,j}=\cO\Big((1+k/j)^2(k^2+j) \Big)$. Note that the complexity is linear only in the regime $k=\cO(\sqrt{j})$, while for $j=\cO(1)$ it is of order~$\cO(k^4)$. 

\begin{rem}
It would be possible to avoid rejection (the repeat loop step) if we had access to a $(j+1)$-to-$(j-1)$ correspondence between the elements of~$\cTe_{k-1,j+1}$ and the elements of~$\cTbe_{k-1,j+1}$. Such a correspondence is provided by Remark~\ref{rem:corresp}. However, it does not seem effective as it requires propagation of bijections starting from an external edge.       
\end{rem}

Similar uniform random samplers can be obtained for $\cB_{k,\ell,j}$ and $\cS_{k,j}$, of respective complexities $\cO\Big((1+k/j)^2(k^2+\ell+j) \Big)$ and $\cO\Big((1+k/j)^3(k^2+j) \Big)$. 

More efficient alternatives are provided by the recursive method of sampling, see e.g.~\cite{BoMo03,BaBu19}. With no rejection involved, they only necessitate simple (rational in terms of the parameters) formulas for the ratio of adjacent coefficients, without actually requiring a positive bijective proof of such formulas. Also, they usually necessitate one more parameter. For instance, letting $\cY_{a,b,c}$ be the set of standard Young tableaux on the 3-line diagram $(a,b,c)$, we have 
\[
\cY_{a,b,c}\simeq \mathbf{1}_{a>b}\cY_{a-1,b,c} \cup \mathbf{1}_{b>c}\cY_{a,b-1,c}\cup \mathbf{1}_{c>0}\cY_{a,b,c-1}.
\]
The induced recursive random sampler $\Gamma Y_{a,b,c}$ thus chooses whether to call  
$\Gamma Y_{a-1,b,c}$ or~$\Gamma Y_{a,b-1,c}$ or~$\Gamma Y_{a,b,c-1}$ with probabilities given by the ratios of the counting coefficients, which are simple rational expressions in~$a$, $b$, $c$, due to massive cancellations of the hook lengths. This gives a linear-time random sampler for~$\cY_{a,b,c}$ and thus for $\cY_{k,j}\simeq\cT_{k,j}$.


\bibliographystyle{alpha}
\bibliography{main}

\end{document}